\documentclass[12pt]{amsart}
\usepackage{epsfig}
\usepackage{wasysym}
\usepackage{tikz}
\usepackage{tikz-cd}
\usepackage{bm}
\usepackage{xcolor}
\usepackage{listings}

\numberwithin{equation}{section}
\usepackage{extpfeil}
\usepackage{array}
\usepackage{adjustbox}
\usepackage{longtable}
\usepackage{makecell}
\usepackage[all]{xy}
\usepackage{amsmath}
\usepackage{amsfonts}
\usepackage{amssymb}
\usepackage{amsthm}
\usepackage{float}
\usepackage[colorlinks=false,urlbordercolor=white]{hyperref}
\usepackage[margin=1in, headsep=10pt]{geometry}

\usepackage[OT2,T1]{fontenc}
\DeclareSymbolFont{cyrletters}{OT2}{wncyr}{m}{n}
\DeclareMathSymbol{\Sha}{\mathalpha}{cyrletters}{"58}

\newtheorem{theorem}{Theorem}
\newtheorem{corollary}[theorem]{Corollary}
\newtheorem{lemma}[theorem]{Lemma}
\newtheorem{remark}[theorem]{Remark}
\newtheorem{prop}[theorem]{Proposition}
\newtheorem{cond}[theorem]{Condition}

\newtheorem{example}[theorem]{Example}

\def\fS{{\mathfrak S}}

\def\bP{{\mathbb P}}
\def\bQ{{\mathbb Q}}

\def\bZ{{\mathbb Z}}

\def\rH{{\mathrm H}}

\def\Pic{\mathrm{Pic}}

\def\Gal{\mathrm{Gal}}
\def\Aut{\mathrm{Aut}}

\def\GL{\mathsf{GL}}

\def\lim{\mathrm{lim}}

\makeatother
\makeatletter

\begin{document}
\title{Potentially stably rational conic bundles over nonclosed fields}
\author{Kaiqi Yang}
\address{Institute of the Mathematical Sciences of the
Americas (IMSA),1365 Memorial Drive, Ungar 515
Coral Gables, FL 33124, USA}
\email{kxy1105@miami.edu}

\date{\today}
\maketitle

\begin{abstract}
We study stable rationality of conic bundles $X$ over $\bP^1$ defined over non-closed field $k$ via the cohomology of the Galois group of finite field extension $k'/k$ with action on the geometric Picard lattice of $X$.
\end{abstract}

\section{Introduction}
\label{sect:intro}

One of the most fundamental yet challenging problems in algebraic geometry is determining whether an algebraic variety $X$ over a field $k$ is {\em rational}, meaning it is birationally equivalent to projective space. A closely related problem is that of {\em stable rationality}, which asks whether $X \times \bP^n$ is rational for some $n$. Despite decades of intense study, these questions still lack fully satisfactory solutions.

One well-known necessary condition for stable rationality in dimensions $\ge 2$ is 

\begin{cond}${\bf{(H1)}}$
	\label{cond:h1}
	\[ 
	\label{eqn:h1}
	\rH^1(G_{k'}, \Pic(\bar{X}))=0, \quad \text{ for all finite extensions} \quad k'/k,
	\] 
\end{cond}
where $\bar{X} = X \times_k \bar{k}$, and $\bar{k}$ is the algebraic closure of base field $k$.

In this paper, we study the (H1) condition for surfaces. 
The first step in investigations of surfaces over non-closed fields is the theory of {\em minimal models}, reducing the problem 
to a list of standard presentations. Geometrically rational surfaces can be birationally modified to 
two types: {\em del Pezzo surfaces} and {\em conic bundles over the projective line}. The modifications proceed via {\em blow-downs} of Galois-orbits consisting of disjoint exceptional curves, until one arrives at {\em minimal} surfaces, which fall in one of the two classes mentioned above. 

In \cite{TY}, we considered the case of del Pezzo surfaces. We 
identified all Galois actions on 
$\Pic(\bar{X})$ satisfying condition (H1). 
One of the by-products of this analysis was the following:

\begin{theorem} 
A stably rational cubic surface over a non-closed field $k$ is rational over $k$. 
\end{theorem}

In this paper, we turn to conic bundles, i.e., fibrations $\pi: X\to \mathbb P^1$, where general fiber is a smooth conic, and there are $n$ degenerate fibers, each consisting of a union of two intersecting lines. In this case, the Galois action
on $\Pic(\bar{X})$
factors through 
the Weyl group $W(\mathsf D_n)$, and our task reduces to investigations of (H1) conditions for various subgroups $G\subseteq W(\mathsf D_n)$.

The main results of this paper are:
\begin{enumerate}
	\item We analyze this invariant for {\em cyclic} subgroups of $W(\mathsf D_n)$ and find that there are three types of actions giving rise to trivial group cohomology.
	\item
	In section \ref{sect:proj}, we show that the relative minimality and (H1) conditions are preserved under projection with respect to orbit decomposition.
	\item
	Based on the result,
	\[
	\mathrm{H}^1(G,\Pic(\bar{X})) \text{ is 2-torsion,}
	\]
	we initiate the analysis of this invariant for noncyclic subgroups of $W(\mathsf{D}_n)$ whose generators are of one of the three identified types, we give 24 classes of groups, up to conjugation, satisfying (H1) and relative minimality conditions. Detailed descriptions of these 24 classes are in section \ref{sect:classes}.
\end{enumerate}

Here is a brief overview of what is known: 
\begin{itemize}
\item Necessary conditions are {\em geometric} (stable) rationality, i.e., (stable) rationality of $X$ over an algebraic closure $\bar{k}$ of $k$, and existence of $k$-rational points.
\item These conditions are also sufficient in dimension one, but not sufficient, in general, in dimensions $\ge 2$.
\item Rationality of surfaces is settled \cite{manin}, \cite{isk}, and can be determined by analyzing the action 
of the absolute Galois group $G_k$ on the set of exceptional curves  (e.g., lines on a cubic surface) 
and the induced action on the geometric Picard group 
$\Pic(\bar{X})$,  but stable rationality remains elusive.  
\end{itemize}

\

\noindent
{\bf Acknowledgments:}  This paper was completed while the author was working in IMSA under the supervision of Ludmil Katzarkov. I am deeply grateful to Brendan Hassett, Ludmil Katzarkov and Yuri Tschinkel for their invaluable advice and insightful discussions,

\section{Geometry and symmetries of conic bundles} 
\label{sect:conic}

From now on, $X$ is a geometrically rational surface over $k$ admitting a standard conic bundle over $k$, i.e., we have
a morphism
$$
\pi: X\to \bP^1
$$
with, geometrically, $n$ degenerate fibers, each a union of two intersecting lines, called {\em exceptional curves}. 
The integer $d=(K_X)^2=8-n$ is called the {\em degree} of the conic bundle $X$.
The two transversal components of the $j$-th degenerate fiber are denoted by $q_j^+$ and $q_j^-$.
The surface is called {\em split} over $k$ if all exceptional curves are defined over $k$.

The free $\mathbb{Z}$-module $\Pic(\bar{X})=\bigoplus_{i=-1}^n\mathbb{Z}l_i$ has rank $(10-d)$ where $d$ is the degree of the conic bundle. The free $\bZ$-module $\Pic(\bar{X})$ has a basis $l_{-1},\cdots,l_n$ where $l_0$ is the class of a fiber of $\pi$, $l_1,\dots,l_n$ are the classes of the components of degenerate fibers, one from each, and the class of canonical divisor $K_X=-2(l_{-1}+l_0)+\sum_{i=1}^nl_i$.

The scalar product on $\Pic(\bar{X})$ is given by
\begin{align*}
&(l_{-1}^2)=(l_{-1},l_i)=(l_0^2)=(l_0,l_i)=0, \quad \text{for }i \geq 1,\\
&(l_{-1},l_0)=1,  \text{ and } (l_i,l_j)=-\delta_{i,j}, \text{ for }i,j\geq 1.
\end{align*}
The module $\Pic(\bar{X})$ has a natural action of the Galois group $\Gal(\bar{k}/k)$ and defines a representation 
\[
\rho:\Gal(\bar{k}/k) \to \Aut_0(\Pic(\bar{X})),
\]
where $\Aut_0$ is the group of automorphisms preserving the scalar product and $K_X$. The splitting group of $X$ is $\text{Im}(\rho)$.

The surface is called {\em k-minimal} if no blow-downs are possible over $k$ and {\em relatively k-minimal} if
\[
\Pic(\bar{X})^{\Gal(\bar{k}/k)}=\mathbb{Z}l_0\oplus\mathbb{Z}K_X.
\]

From \cite{kst}, a relatively k-minimal but non-k-minimal conic bundle can appear only if its degree is equal to 3,5,6,7 or 8.
When the degree of a conic bundle is less than 3, relative k-minimality is the same as k-minimality.

The automorphism group of $n$ pairs of symbols $j^+, j^-, j=1,\dots,n$ is the Weyl group 
$$
W(\mathsf B_n)\cong \fS_n \rtimes (\mathbb{Z}/2)^n,
$$ 
the signed permutation group, 
where the symmetric group $\fS_n$ acts on $j^{\pm}$ by permuting the index $j$ and $(\mathbb{Z}/2)^n$ is 
generated by $c_1,\ldots,c_n$, where $c_j$ exchanges $j^+$ and $j^-$ and leaves the remaining symbols invariant.

\begin{lemma}
    We have the following equation
    \begin{align*}
        \tau c_j=c_{\tau(j)}\tau,\quad \tau \in \fS_n,
    \end{align*}
    or equivalently,
    \begin{align*}
        \tau c_{\tau^{-1}(j)}=c_j\tau.
    \end{align*}
\end{lemma}
\begin{proof}
    We have
    \begin{align*}
        \tau c_j(j^{\pm})=\tau(j^{\mp})=(\tau(j))^{\mp}=c_{\tau(j)}(\tau(j)^{\pm})=c_{\tau(j)}\tau(j^{\pm}).
    \end{align*}
    For other cases, the equation holds trivially.
\end{proof}

We denote the projection from $W(\mathsf{B}_n)$ to $\mathfrak{S}_n$ as
\begin{align*}
    pr:W(\mathsf{B}_n) \to \mathfrak{S}_n,\quad 
    pr(c_{j_1}\cdots c_{j_t}\tau)=\tau,\quad \tau \in \mathfrak{S}_n.
\end{align*}

There is a character 
$$
\sigma:W(\mathsf B_n) \to \{\pm 1\},
\quad\quad 
\sigma(c_{j_1}c_{j_2}\cdots c_{j_t}\tau)=(-1)^t, \quad \tau \in \fS_n,
$$ 
with kernel 
$$
W(\mathsf D_n) \cong \fS_n \rtimes (\mathbb{Z}/2)^{n-1}.
$$
The group $W(\mathsf D_n)$ is the automorphism group acting on $n$ pairs of degenerate fibers of a standard conic bundle.

There is a natural integral representation of $W(\mathsf B_n)$,
$$
\Psi:W(\mathsf{B}_n) \to \GL_n(\mathbb{Z}),
$$
given by
\[
(\Psi(c_{j_1}\cdots c_{j_t}\tau))_{uv}=\sum_{i=1}^ns(i)\delta_{u,i}\delta_{v,\tau(i)},
\]
where the sign function $s(i)$ is given by
\[
s(i)=
\begin{cases} 
-1 & \text { $i \in \{j_1,\cdots ,j_t\}$},\\
1  &  \text{otherwise}.
\end{cases}
\]

The action of the automorphism group $W(\mathsf D_n)$ on $\Pic(\bar{X})$ with $n\geq 4$, induces an integral representation 
$$
\Phi:W(\mathsf D_n) \to \GL_{n+2}(\mathbb{Z}), \quad n\ge 4. 
$$
This representation $\Phi$ could be derived from the representation $\Psi$ of $W(\mathsf B_n)$ restricted to its subgroup $W(\mathsf D_n)$.

Assume a group element $g \in W(\mathsf D_n)$ has the form $g=c_{j_1}\cdots c_{j_t} \tau$, where $t$ is an even integer, the image $\Phi(g) \in \GL_{n+2}(\bZ)$ is given by 
\begin{align*}
    \Phi(g)=
    \begin{pmatrix}
    A_{g} & B_{g}\\
    C_{g} & \Psi(g)
\end{pmatrix},
\end{align*}
where
where $A_g$ is a $2 \times 2$, $B_g$ is a $2 \times n$, and $C_g$ is an $n \times 2$ matrix.

In detail, 
\begin{align*}
B_g=
\begin{pmatrix}
0 & 0 & \cdots & 0\\
b'_g(1) & b'_g(2) & \cdots & b'_g(n)
\end{pmatrix},
\end{align*}
where 
\[
b'_g(i)=
\begin{cases}
1 \quad &\text{if the nonzero element in the $i$-th column of $\Psi(g)$ is $-1$ },\\
0 \quad &\text{if the nonzero element in the $i$-th column of $\Psi(g)$ is $1$ },
\end{cases}
\]
\begin{align*}
C_g=
\begin{pmatrix}
c'_g(1) & 0\\
c'_g(2) & 0\\
\vdots & \vdots\\
c'_g(n) & 0
\end{pmatrix},
\end{align*}
where
\[
c'_g(i)=
\begin{cases}
-1 \quad &\text{if the nonzero element in the $i$-th row of $\Psi(g)$ is $-1$ },\\
0 \quad &\text{if the nonzero element in the $i$-th row of $\Psi(g)$ is $1$ },
\end{cases}
\]
and
\begin{align*}
A_g=
\begin{pmatrix}
1 & 0\\
\frac{1}{2}(\sum_{i=1}^n(b'_g(i))) & 1
\end{pmatrix}.
\end{align*}

Group element has the form $g=c_{j_1}\cdots c_{j_t}\tau$ with $t$ is even, thus  $\sum_{i=1}^n(b'_g(i))=t$ is divisible by $2$.

In the rest of the paper, if there is no confusion, we write $g$ for $\Phi(g)$ for simplicity.

We adopt notation from \cite{kst}:
a group element $g\in W(\mathsf D_n)$ can be expressed as
$$
g=\beta_1\cdots \beta_R,
$$
where each $\beta$ is a signed permutation cycle, defined as a product of a single permutation cycle $\tau$ and $c_i$'s with indices $i$ lie in cycle $\tau$.

\section{Group cohomology}
A finite group $G$ has a presentation with $m$ generators subject to relations, written as $G=\langle g_1,\dots,g_m \rangle$. The group cohomology discussed in (H1) condition is the quotient of 1-cocycles by 1-coboundaries,
\begin{align*}
    \rH^1(G,\Pic(\bar{X}))=Z^1(G,\Pic(\bar{X}))/B^1(G,\Pic(\bar{X})).
\end{align*}
where $Z^1(G,\Pic(\bar{X}))$ is the $\bZ$-module of 1-cocycles which are maps $f:G \to \Pic(\bar{X})$ such that

\begin{align*}
    f(gh) = f(g) + g \cdot f(h)\quad \forall g,h \in G,
\end{align*}
and $B^1(G,\Pic(\bar{X}))$ is the $\bZ$-module of 1-coboundaries generated by
\begin{align*}
    f(g) = g \cdot a - a\quad  \forall g \in G, a \in \Pic(\bar{X}).
\end{align*}

Given a defining relation of $G$,
\begin{align*}
    \prod_{k=1}^L g_{i_k}=1
\end{align*}
where $i_k \in \{1,\cdots,m\}$, it induces a linear relation on 1-cocycle $f$,
\begin{align*}
    f(g_{i_1})+\sum_{t=1}^{L-1}(\prod_{k=1}^{t}g_{i_k})f(g_{i_{t+1}})=0,
\end{align*}
and we regard $f$ as a length $m(n+2)$ column that records $f(g_i)$ for $i=1,\cdots,m$. All defining relations of $G$ produce a matrix $M$ and the $\bZ$-module $Z^1(G,\Pic(\bar{X}))$ is the kernel of $M$.

We record two theorems:
\begin{itemize}
    \item From Theorem (10.3) in \cite{coho_book}, we have
    \[
    \rH^1(G,\Pic(\bar{X}))=\bigoplus_{p}\rH^1(G_p,\Pic(\bar{X}))^G,
    \] 
    where $G_p$ is the Sylow $p$-subgroup of group $G$.
    \item From \cite{ASENS_1981_4_14_1_41_0,colliotthlne2017surfaces}
    $$
    \mathrm{H}^1(G,\Pic(\bar{X})) \text{ is 2-torsion.}
    $$
\end{itemize}

From the above two theorems, we conclude
\begin{lemma}
\label{lem:sylow2}
    Finite group $G$ satisfies $\mathrm{(H1)}$ condition if and only if its Sylow 2-subgroup $G_2$ satisfies $\mathrm{(H1)}$ condition.
\end{lemma}
\begin{proof}
    If finite group $G$ satisfies (H1) condition,
    \begin{align*}
        \rH^1(H,\Pic(\bar{X}))=0\quad \forall\ H \subseteq G,
    \end{align*}
    then the Sylow 2-subgroup $G_2$, as a subgroup of $G$, satisfies (H1) condition.
    
    Conversely, if Sylow 2-subgroup $G_2$ satisfies (H1) condition, then for any subgroup $H \subseteq G$, we have $H_2 \subseteq G_2$ and
    \begin{align*}
    \rH^1(H,\Pic(\bar{X}))=\rH^1(H_2,\Pic(\bar{X}))^{H}=0,
    \end{align*}
\end{proof}

To match the index set of the basis of Picard lattice, the index set of length $(n+2)$ column is from $-1$ to $n$.

The coboundaries module $B^1(G,\Pic(\bar{X}))$ are spanned by elements $f_i$ of length $m(n+2)$, with $i\in \{-1,\dots,n\}$, 
\begin{align*}
    B^1(G,\Pic(\bar{X}))=\bZ\langle f_{-1},f_1,\dots,f_n\rangle.
\end{align*}
Each $f_i$ has the following form:
\begin{align*}
f_i=[((g_1-\mathbb{I})e_{i})^T,\dots,((g_m-\mathbb{I})e_{i})^T]^T,
\end{align*}
where $i=-1,\cdots,n$ and $e_{i}$ is a length $(n+2)$ $\bZ$-module element with entries
$$
(e_{i})_l=\delta_{i,l},\ l=-1,\cdots,n.
$$

We use notation 
\begin{align*}
f_{i,g}=(g-\mathbb{I})e_{i},\ i\in \{-1,\dots,n\}.
\end{align*}

The special element $f_{-1}$ satisfies the relation,
\begin{align*}
    f_{-1}=\frac{1}{2}\sum_{i=1}^nf_i,
\end{align*}
and we focus on the submodule $\bZ\langle f_1,\dots,f_n \rangle \subseteq B^1(G,\Pic(\bar{X}))$.

We have
\begin{align*}
\rH^1(G,\Pic(\bar{X}))&=Z^1(G,\Pic(\bar{X}))/B^1(G,\Pic(\bar{X}))\\
&=(Z^1(G,\Pic(\bar{X}))/\bZ\langle f_1,\dots,f_n\rangle)/(\bZ\langle f_{-1},f_1,\dots,f_n\rangle /\bZ\langle f_1,\dots,f_n\rangle).
\end{align*}

If $f_{-1} \notin \bZ \langle f_1,\dots,f_n\rangle $, then
\begin{align*}
\bZ\langle f_{-1},f_1,\dots,f_n\rangle/\bZ\langle f_1,\dots,f_n\rangle\cong \mathbb{Z}/2,
\end{align*}
otherwise,
\begin{align*}
\rH^1(G,\Pic(\bar{X}))\cong Z^1(G,\Pic(\bar{X}))/\bZ\langle f_1,\dots,f_n\rangle.
\end{align*}

If $\rH^1(G,\Pic(\bar{X})) \neq 0$, then there exists an element $\xi \in Z^1(G,\Pic(\bar{X}))$ such that 
$$
\xi \notin B^1(G,\Pic(\bar{X})) \text{ and } 2\xi \in B^1(G,\Pic(\bar{X})),
$$ 
that implies we can find integers $\{\alpha_i\}_{i=-1}^n$ such that
$$
\sum_{i=-1}^{n}\alpha_if_i=2\xi, \quad \alpha_i \in \mathbb{Z}.
$$
Since 
$$
[\xi]=[\xi\pm f_i]\in \rH^1(G,\Pic(\bar{X})),\ \forall i\in \{-1,\dots,n\},
$$
and
\[
\sum_{i=-1}^n\alpha_if_i\pm2f_j=2(\xi\pm f_j),
\]
it follows that for a nonzero element $[\xi] \in \rH^1(G,\Pic(\bar{X}))$, we can find an index set $I$ such that
\[
\sum_{i\in I}f_i=2\xi',
\] 
for some $\xi'$ and $[\xi]=[\xi'] \in \rH^1(G,\Pic(\bar{X}))$.

Conversely, if there is an index set $I$ such that 
\begin{align*}
    \frac{1}{2}\sum_{i \in I}f_i = \xi \notin B^1(G,\Pic(\bar{X})),
\end{align*}
where $\xi$ is integral, then 
\begin{align*}
    M (\sum_{i \in I}f_i)=M(2\xi)=2M\xi=0,
\end{align*}
which implies $M\xi=0$, $\xi \in Z^1(G,\Pic(\bar{X}))$ and $0 \neq [\xi] \in \mathrm{H}^1(G,\Pic(\bar{X}))$.

\begin{lemma}
\label{lem:No-1}
	If $\frac{1}{2}\sum_{i \in I}f_i=\xi$ for some index set $I$ and $0 \neq [\xi] \in \mathrm{H}^1(G,\Pic(\bar{X}))$, then $-1 \notin I$.
\end{lemma}
\begin{proof}
        We prove this lemma by contradiction.
        
        If $-1 \in I$, let $I'=I\setminus \{-1\}$, then we have 
	\[
	f_{-1}+\sum_{i \in I'}f_i=2\xi,
	\]
	which implies
	\[
	f_{-1}+\sum_{i \in I'}f_i=0 \mod 2,
	\]
	or equivalently,
	\[
	f_{-1}-\sum_{i \in I'}f_i=0 \mod 2.
	\]
	Element $f_{-1}$ satisfies the following relation,
	\[
	f_{-1}=\frac{1}{2}\sum_{i=1}^nf_i.
	\] 
	We substitute the expression of $f_{-1}$ into the previous equation and we have
	\[
	\sum_{i \notin I'}f_i-\sum_{i \in I'}f_i=0 \mod 4.
	\]
	When restricted to a generator $g_l$, we have 
    $$
    \sum_{i\geq 1,i \notin I'}f_{i,g_l}-\sum_{i \geq 1,i \in I'}f_{i,g_l}=0 \mod 4, \quad \forall l=1,\dots,m.
    $$
	The entries of this element of index $1$ to $n$ have 3 possibilities, $\pm2$ and 0.
	Thus all entries of index $1$ to $n$ are 0 restricted to every generator $g_l$.

    The entry of index 0, which is $\frac{1}{2}$ times the total number of minus signs in the index set $I$ and $I'$, are the same.
    
	We conclude that
	\[
	\sum_{i \notin I'}f_i=\sum_{i \in I'}f_i
	\]
	and substitute it to the expression of $f_{-1}$, we have
	\[
	f_{-1}=\frac{1}{2}\sum_{i=1}^nf_i=\frac{1}{2}(\sum_{i \notin I'}f_i+\sum_{i \in I'}f_i)=\sum_{i \in I'}f_i
	\]
	and
	\[
	2f_{-1}=2\xi,
	\]
	which implies
	\[
	f_{-1}=\xi=\sum_{i \in I'}f_i \in B^1(G,\Pic(\bar{X})),
	\]
    and 
    \[
    [\xi]=[f_{-1}]=0 \in\rH^1(G,\Pic(\bar{X})),
    \]
    which is a contradiction to our assumption.
\end{proof}

\begin{lemma}
\label{lem:indorb}
	Assume $\xi = \frac{1}{2}\sum_{i \in I}f_i$ for some index set $I$, $0 \neq [\xi] \in \mathrm{H}^1(G,\Pic(\bar{X}))$ and up to conjugation, assume a group element $g=\beta_1\beta_2\dots\beta_R$ with $\beta_1=c_{j_1}\dots c_{j_{t}}(1\ \dots \ w)$. If $1\in I$, then $\{1,\dots,w\}\subset I$. 
\end{lemma}
\begin{proof}
	Let $s(i)$ be the sign function of the signed permutation cycle $\beta_1$. Under the assumption, we have
	\begin{align*}
	g-\mathbb{I}=
	\left(
	\small
	\begin{array}{cc|ccccc}
	0 & 0 & 0 & 0 & \dots & 0 & *\\
	* & 0 & \frac{1-s(w)}{2} & \frac{1-s(1)}{2} & \dots & \frac{1-s(w-1)}{2} & *\\
	\hline
	\frac{s(1)-1}{2} & 0 & -1 & s(1) & \dots & 0 & *\\
	\frac{s(2)-1}{2} & 0 & 0 & -1 & \dots & 0 & *\\
	\vdots & \vdots & \vdots & \vdots & \ddots & \vdots & \vdots\\
	\frac{s(w)-1}{2} & 0 & s(w) & 0 & \dots & -1 & *\\
	* & * & * & * & * & * & *
	\end{array}
	\right),
	\end{align*}
	index $1\in I$ and 
	\[
	\sum_{i \in I}f_i=0 \mod 2.
	\]
	From Lemma \ref{lem:No-1}, since $-1 \notin I$, the only column with a nonzero entry at index 1, other than $f_1$, is $f_2$, therefore $2 \in I$.
	Use similar arguments, we conclude $\{1,\dots,w\} \subseteq I$.
\end{proof}

\begin{remark}
In summary, if an index $i\in I$, then all indices in the same signed permutation cycle as $i$ are also in the index set $I$. By applying this lemma to each signed permutation cycle in every generator of the group $G$, the index set $I$ is a union of orbits of $pr(G)$ acting on $\{1,\cdots,n\}$.
\end{remark}

From the discussion above, to identify an obstruction to the triviality of group cohomology of $G$, we need to find index sets $I$ such that
\begin{align*}
    \xi=\frac{1}{2}\sum_{i \in I}f_i \text{ and } 0\neq [\xi] \in \mathrm{H}^1(G,\Pic(\bar{X})),
\end{align*}
which implies
\begin{align*}
     \frac{1}{2}\sum_{i \in I}f_i = \xi \notin \bZ\langle f_1,\dots,f_n\rangle.
\end{align*}
Equivalently, it means there is no integer solution $\{\alpha_{j,I}\}$ to
\begin{align*}
    \frac{1}{2}\sum_{i \in I}f_i=\sum_{j=1}^n \alpha_{j,I} f_j.
\end{align*}

When restricted to a generator $g$, we study the following linear equation and find integer solutions of
\begin{align*}
    \xi_g=\frac{1}{2}\sum_{i \in I}f_{i,g}=\sum_{j=1}^n\alpha_{j,I,g}f_{j,g}.
\end{align*}

Denote the $\bZ$-module of all solutions $\{\alpha_{j,I,g}\}$ to this linear equation as $V_{I,g}$, we have the following statement:

\begin{align*}
    \cap_{i=1}^m V_{I,g_i} \ne \emptyset, \text{ with } G=\langle g_1,\dots,g_m \rangle
    \Longleftrightarrow
    \frac{1}{2}\sum_{i \in I}f_i \in \bZ\langle f_1,\dots,f_n\rangle.
\end{align*}

Now we reduce the study of group cohomology of $G$ to the study of $V_{I,g}$, where $g$ is a generator of $G$ and $I$ is an index set which is a union of orbits of $pr(G)$ acting on $\{1,\cdots,n\}$. 

\section{Computing group cohomology of cyclic groups}
\label{sect:cyclic}

Assume $g \in W(\mathsf{D}_n)$ has the form
$$
g=\beta_1\cdots\beta_R.
$$

From Lemma \ref{lem:indorb}, orbit decomposition of the cyclic group $pr(\langle g \rangle)$ is determined by signed permutation cycles $\beta$'s.  We focus on one signed permutation cycle $\beta_1$.

Up to conjugation, assume
\begin{align*}
    \beta_1=c_{j_1}\dots c_{j_{t}}(1\ \dots \ w),
\end{align*}
and there are several cases:
\begin{itemize}
    \item If $w \geq 2$ and $t$ is even, equal to $2h$, then
    \begin{align*}
        \frac{1}{2}\sum_{i=1}^wf_i=\sum_{u=1}^h\sum_{r=(j_{2u-1})+1}^{j_{2u}}f_{r} \in \bZ \langle f_1,\cdots,f_n\rangle
    \end{align*}
    \item If $t$ is odd, then
    \begin{align*}
        \sum_{i=1}^wf_i \neq 0 \mod 2,
    \end{align*}
    since the entry of index $0$ is $\sum_{i=1}^w \frac{1-s(i)}{2}=t$, an odd integer.
\end{itemize}

From the discussion above, we focus on signed permutation cycles with an odd number of minus signs. Without loss of generality, we denote them as $\{\beta_1,\cdots,\beta_{\Lambda}\}$ and $I_i \subseteq \{1,\cdots,w\}$ is the set of indices appearing in $pr(\beta_i)$. The number of such signed permutation cycles, $\Lambda$, is an even integer since $g \in W(\mathsf D_n)$.

\begin{itemize}
    \item If $\Lambda=0$, meaning there are no signed permutation cycles with an odd number of minus signs, then we have
    \begin{align*}
        Z^1(\langle g\rangle,\Pic(\bar{X}))/\bZ \langle f_1,\cdots,f_n\rangle=0
    \end{align*}
    and
    \begin{align*}
        \mathrm{H}^1(\langle g\rangle,\Pic(\bar{X}))=0
    \end{align*}
    \item If $\Lambda \geq 2$, then for any set $\{i_1,\cdots,i_{2h}\} \subseteq \{1,\cdots,\Lambda\}$, write $I=I_{i_1} \cup \cdots \cup I_{i_{2h}}$. Without loss of generality, we assume $i_1=1$ and
    \begin{align*}
        \beta_1=c_{j_1}\dots c_{j_{t}}(1\ \dots \ w).
    \end{align*}
    We have
    \begin{align*}
        \xi=\frac{1}{2}\sum_{i\in I}f_i \in Z^1(\langle g\rangle,\Pic(\bar{X})),
    \end{align*}
    and 
    \begin{align*}
    (\xi)_u=\frac{-1+s(u)}{2}=
    \begin{cases}
        -1, &u\in \{j_1,\cdots,j_t\},\\
        0, &\text{otherwise},
    \end{cases}
\end{align*}
for $1\leq u\leq w$.

The linear system,
\begin{align*}
    \xi =\sum_{i}\alpha_{i,I}f_i
\end{align*}
is equivalent to the following,
\begin{align*}
    &(\xi)_u=-\alpha_{u,I}+\alpha_{u+1,I}s(u),\ u=1,\cdots,w-1\\
    &(\xi)_w=-\alpha_{w,I}+\alpha_{1,I}s(w).
\end{align*}

If $s(u)=1$, then
\begin{align*}
    0=-\alpha_{u,I}+\alpha_{u+1,I} \Rightarrow \alpha_{u,I}=\alpha_{u+1,I},
\end{align*}
otherwise,
\begin{align*}
    1=\alpha_{u,I}+\alpha_{u+1,I}.
\end{align*}

From the above discussion, assuming $s(w)=1$, we have
\begin{align*}
    &\alpha_{1,I}=\cdots=\alpha_{j_1,I}=\alpha,\\
    &\alpha_{j_1+1,I}=\cdots=\alpha_{j_2,I}=1-\alpha,\\
    &\alpha_{j_2+1,I}=\cdots=\alpha_{j_3,I}=\alpha,\\
    &\cdots\\
    &\alpha_{j_t,I}=\cdots=\alpha_{w,I}=1-\alpha.
\end{align*}
Since $s(w)=1$, we have $1-\alpha=\alpha_{w,I}=\alpha_{1,I}=\alpha$, which implies $\alpha=\frac{1}{2}$, not an integer.

We conclude that $\xi \notin \bZ \langle f_1,\cdots,f_n\rangle$, and index set $I$ induces a nonzero element in
$$
Z^1(\langle g\rangle,\Pic(\bar{X}))/\bZ \langle f_1,\cdots,f_n\rangle.
$$

For the case where $s(w)=-1$, the argument is similar, leading to the same conclusion. Furthermore
\begin{align*}
    f_{-1}=\frac{1}{2}\sum_{i=1}^nf_1 \notin \bZ \langle f_1,\cdots,f_n\rangle.
\end{align*}

We have $\Lambda-1$ index sets, $I_1 \cup I_j$, $j=2,\cdots,\Lambda$, which correspond to nonzero elements and form a generating set of 
\begin{align*}
    Z^1(\langle g\rangle,\Pic(\bar{X}))/\bZ \langle f_1,\cdots,f_n\rangle=(\bZ/2)^{\Lambda-1}.
\end{align*}
Since $f_{-1} \notin \bZ \langle f_1,\dots,f_n\rangle$, we have
\begin{align*}
    \bZ \langle f_{-1},f_1,\cdots,f_n\rangle/\bZ \langle f_1,\cdots,f_n\rangle = \bZ/2, 
\end{align*}
and thus
\begin{align*}
    \mathrm{H}^1(\langle g\rangle,\Pic(\bar{X}))=(\bZ/2)^{\Lambda-2}.
\end{align*}
\end{itemize}

\begin{prop}
\label{h1:cyclic}
Let $g\in W(\mathsf D_n)$ with $g=\beta_1\cdots\beta_R$, and $\Lambda$ denote the number of signed permutation cycles $\beta_i$ with $\sigma(\beta_i)=-1$, then 
\begin{align*}
\mathrm{H}^1(\langle g\rangle,\Pic(\bar{X}))=
\begin{cases}
0, &\quad \text{if $\Lambda=0,2$},\\
(\mathbb{Z}/2)^{\Lambda-2}, &\quad \text{otherwise}.
\end{cases}
\end{align*}
\end{prop}

\begin{corollary}
\label{Genform}
If a cyclic group $G=\langle g\rangle$ satisfies the $\mathrm{(H1)}$ condition, i.e.,
$$
\mathrm{H}^1(\langle g^i\rangle,\Pic(\bar{X}))=0 \quad \forall \ i,
$$
and $\beta$ is a signed permutation cycle in $g$ with $\sigma(\beta)=-1$, then the length of the cycle $pr(\beta) \leq 2$ and there is at most one such signed permutation cycle.

Up to conjugation, there are three types of generator s$g$ satisfying the assumption:
\begin{itemize}

\item[(1)] $g=c_1c_2\beta_3\cdots\beta_R$,

\item[(2)] $g=c_1c_2(2\ 3)\beta_3\cdots\beta_R$,

\item[(3)] $g=\beta_1\cdots\beta_R$,
\end{itemize}
where $\sigma(\beta_i)=1$ for all $i$.
\end{corollary}
\begin{proof}
	
Without loss of generality, assume there is a signed permutation cycle $\beta$  has the form
$\beta=c_{j_1}\ldots c_{j_t}(1,\cdots,w)$ and $\sigma(\beta)=-1$, meaning $t$ is odd.

Taking its $w$-th power, we have $w$ signed permutation cycles $c_j$ for $j=1,\cdots,w$.

If $w \geq 3$, then the number of signed permutation cycles $\beta$ with $\sigma(\beta)=-1$, denoted as $\Lambda$, is greater or equal to $3$. Therefore, by Proposition \ref{h1:cyclic}, it follows that the group cohomology is nontrivial.

For a length two signed permutation cycle with one minus sign, up to conjugation, the only possibility is
\[
c_1(1,2).
\]

Suppose there are at least two such signed permutation cycles in $g$, without loss of generality, we may assume
\[
g=c_1(1,2)c_3(3,4)\beta_3\cdots\beta_R.
\]

Taking the square of $g$, we have four signed permutation cycles $c_1,c_2,c_3,c_4$, 
which implies nontrivial group cohomology by Proposition \ref{h1:cyclic}.
\end{proof}

\begin{theorem} \label{thm:13}
	Let $G \subset W(\mathsf{D}_n)$ be a subgroup satisfying both the $\mathrm{(H1)}$ and the minimality condition, then the $G$-action on the set of degenerate fibers $\{q_j^{\pm}\}$ has at most three orbits.
\end{theorem}
\begin{proof}
	
	Consider the $G$-action on the set of degenerate fibers $\{q_1^{\pm},\dots,q_n^{\pm}\}$.
	By assumption, $G \subset W(\mathsf{D}_n)$ is minimal, thus every pair of degenerate fibers $q_j^{+}$ and $q_j^-$ are in the same orbit. Thus we can find a group element $g\in G$ such that $g$ interchanges $q_j^+$ and $q_j^-$ which implies that $g$ contains the signed permutation cycle $c_j$. 
	Therefore, for every orbit $O$ and every pair of degenerate fibers $\{q_j^{\pm}\} \subset O$, we can find a $g\in G$ such that $g=c_j\beta_2\dots\beta_R$, where the signed permutation cycles $\beta_2,\dots,\beta_R$ are disjoint and do not contain index $j$.
	Note that $g$ contains a signed permutation cycle $c_j$ and $g \in W(\mathsf{D}_n)$, thus there is another signed permutation cycle with odd number of minus signs; then $g$ has the form of case (2) or case (3) described in Corollary \ref{Genform}. 
	
	Now we assume there are at least four orbits and denote them as $O_1,O_2,O_3,O_4$.
    
    We mainly focus on two signed permutation cycles with one minus sign.
	We can pick an element $g_1=c_{j_1}c_{j_2}*$, where $q_{j_1}^{\pm}$ lies in $O_1$; $g_2=c_{j_3}c_{j_4}*$ where $q_{j_3}^{\pm}$ lies in $O_2$; $g_3=c_{j_5}c_{j_6}*$, where $q_{j_5}^{\pm}$ lies in $O_3$; $g_4=c_{j_7}c_{j_8}*$, where $q_{j_7}^{\pm}$ lies in $O_4$. Here we use $*$ to represent some signed permutation cycle with irrelevant signs and not containing the $j_1,j_3,j_5,j_7$, respectively.
	\begin{itemize}
		\item If $q_{j_2}^{\pm}$ lies in $O_1$:
		
		Since cyclic groups generated by $g_1g_2$, $g_1g_3$, $g_1g_4$ respectively should have trivial group cohomology, then $q_{j_4}^{\pm}, q_{j_6}^{\pm}, q_{j_8}^{\pm}$ have to be in the same orbit $O_1$.
		We note that $q_{j_3}^{\pm}$ lie in $O_2$, $q_{j_5}^{\pm}$ lies in $O_3$, $q_{j_7}^{\pm}$ lies in $O_4$, thus in the group element $g_2g_3g_4$, there are three signed permutation cycles with one minus sign, one from each orbit, $O_2,O_3,O_4$. Then $g_2g_3g_4$ is an element with at least four signed permutation cycles of one minus sign and the cyclic group, generated by $g_2g_3g_4$, is an obstruction to the (H1) condition.
		
		\item If $q_{j_2}^{\pm}$ lies in another orbit, and we may assume it lies in orbit $O_2$:
		
		Cyclic groups generated by $g_1g_3$ and $g_1g_4$ respectively should have trivial group cohomology, that implies 
		$q_{j_6}^{\pm}$ and $q_{j_8}^{\pm}$ lie in orbit $O_1$ or $O_2$.
		Since the cyclic group generated by
		$g_3g_4$ has trivial group cohomology, $q_{j_6}^{\pm}$ and $q_{j_8}^{\pm}$ lie in the same orbit, either in $O_1$ or in $O_2$, and $g_3g_4$ has two signed permutation cycles with one minus sign, one from orbit $O_3$ and the other from orbit $O_4$. The element $g_1g_3g_4$ contains four signed permutation cycles with one minus sign, one from each orbit $O_1,O_2,O_3,O_4$, the cyclic group, generated by $g_1g_3g_4$, is an obstruction to the (H1) condition.
	\end{itemize}
	
	From the discussion above, we conclude that there are at most three orbits. 
\end{proof}

\section{Computing $\mathrm{H}^1(G,\Pic(\bar{X}))$ for general group $G$}
\label{sect:noncyclic}

Assume group $G$ has $m$ generators, $G=\langle g_1,\cdots,g_m\rangle$. Since we are interested in group $G$ having (H1) condition, we further assume that all group elements have the form described in corollary \ref{Genform}.

For a finite abelian group $G$, our assumption already implies (H1) condition of $G$.

\begin{prop}
    \label{prop:h1bicyc}
    If $G$ is a finite abelian group with actions on $\Pic(\bar{X})$ for some conic bundle $X$ and for any $g\in G$, 
    \begin{align*}
        \mathrm{H}^1(\langle g\rangle,\Pic(\bar{X}))=0,
    \end{align*}
    then $G$ satisfies (H1) condition.
\end{prop}
\begin{proof}
    For any finite abelian group $G$, denote $d(G)$ the minimum of cardinalities of all generating sets of group $G$.
    
    We prove this lemma by induction on $d(G)$.  
    
    Since for any cyclic subgroup  $H=\langle h\rangle \subseteq G$, i.e., $d(H)=1$, by assumption, we have
    \begin{align*}
        \mathrm{H}^1(\langle h\rangle,\Pic(\bar{X}))=0,
    \end{align*}
    and it is easy to check by proposition \ref{h1:cyclic}.

    Assume for any subgroup $H$ of $G$ with $d(H) \leq n$, we have
    \begin{align*}
        \mathrm{H}^1(H,\Pic(\bar{X}))=0.
    \end{align*}
    For any subgroup $H \subseteq G$ with $d(H)=n+1$, i.e., there are elements $h_1,\dots,h_{n+1}$ such that
    \begin{align*}
        H=\langle h_1 \rangle \times \cdots \times \langle h_{n+1} \rangle,
    \end{align*}
    we denote $H'=\langle h_2 \rangle \times \cdots \times \langle h_{n+1} \rangle$, then $d(H')\leq n$.
    
    There is a short exact sequence,
    \begin{align*}
        0 &\to \bigoplus_{i=0}^1(\mathrm{H}^i(\langle h_1\rangle,\Pic(\bar{X}))\otimes \mathrm{H}^{1-i}(H',\Pic(\bar{X}))) \to \mathrm{H}^1(H,\Pic(\bar{X}))\\
        &\to \bigoplus_{i=0}^2 \mathrm{Tor}_{\bZ}(\mathrm{H}^i(\langle h_1\rangle,\Pic(\bar{X})),\mathrm{H}^{2-i}(H',\Pic(\bar{X}))) \to 0.
    \end{align*}
    By assumption,
    \begin{align*}
        \mathrm{H}^1(\langle h_1 \rangle ,\Pic(\bar{X}))=0,\quad \mathrm{H}^1(H',\Pic(\bar{X}))=0,
    \end{align*}
    and
    \begin{align*}
        \mathrm{H}^0(\langle h_1 \rangle ,\Pic(\bar{X}))=\Pic(\bar{X})^{\langle h_1 \rangle},\quad \mathrm{H}^0(H' ,\Pic(\bar{X}))=\Pic(\bar{X})^{H'},
    \end{align*}
    are submodules of the free $\bZ$-module $\Pic(\bar{X})$, are free, then
    \begin{align*}
        \bigoplus_{i=0}^2 \mathrm{Tor}_{\bZ}(\mathrm{H}^i(\langle h_1\rangle,\Pic(\bar{X})),\mathrm{H}^{2-i}(H',\Pic(\bar{X})))=0,
    \end{align*}
    and
    \begin{align*}
        \mathrm{H}^1(H,\Pic(\bar{X}))\cong \bigoplus_{i=0}^1(\mathrm{H}^i(\langle h_1\rangle,\Pic(\bar{X}))\otimes \mathrm{H}^{1-i}(H',\Pic(\bar{X})))=0.
    \end{align*}
\end{proof}
\begin{remark}
   Proposition \ref{prop:h1bicyc} verifies the (H1) condition for all classes in section \ref{sect:classes} except classes 11, 18, 19, 22, and 24. Discussions of the (H1) condition for those 5 classes are in Examples \ref{exam:caseD4} and \ref{exam:case19}.
\end{remark}

For a non-abelian group $G$, we study $V_{I,g}$ with $g$ a generator of $G$ and index set $I$, a union of orbits of $pr(G)$. We consider the following cases,

\begin{itemize}		
	\item Case (1):
	
	If a generator $g=c_1c_2\beta$ where $\beta$ is a signed permutation cycle that does not include 1,2 and $1\in I$, then
	\[
	g-\mathbb{I}=
	\left(
	\small
	\begin{array}{cc|ccc}
	0 & 0 & 0 & 0 & *\\
	* & 0 & 1 & 1 & *\\
	\hline
	-1 & 0 & -2 & 0 & *\\
	-1 & 0 & 0 & -2 & *\\
	* & 0 & 0 & 0 & *\\
	\vdots & \vdots & \vdots & \vdots & \vdots\\
	* & 0 & 0 & 0 & *
	\end{array}
	\right)
	\]
    and the index set $I$ should satisfy the condition
    
	\begin{align*}
	\sum_{i\in I}f_{i,g} \equiv 2\xi_g \equiv 0 \mod 2.
	\end{align*}
    Consider the entry of index 0, since $-1 \notin I$ by Lemma \ref{lem:No-1}, and all other signed permutation cycles have an even number of minus signs, we must have $1,2\in I$. Therefore, we have the following equation:
	\[
	f_{1,g}+f_{2,g}+\sum_{i \in I-\{1,2\}}f_{i,g}=2[0,*,-1,-1,*,\dots,*]^T=2\xi_g.
	\]
	
    By Comparing the entries of index 1 and 2 of $\xi_g$, we conclude that $\xi_g \notin \bZ \langle f_{1,g},\dots,f_{n,g}\rangle$. Consequently, we have
    $$
    \frac{1}{2}\sum_{i \in I} f_i = \xi \notin \bZ \langle f_{1},\dots,f_{n}\rangle,
    $$
 
    which gives a nonzero element in    
    $Z^1(G,\Pic(\bar{X}))/\bZ\langle f_{1},\dots,f_{n}\rangle$.
	\item Case (2):
	
	If a generator $g=c_1c_2(2\ 3)\beta$ where $\beta$ is a signed permutation cycle that does not include 1,2,3 and $1 \in I$ then
	\[
	g-\mathbb{I}=
	\left(
	\small
	\begin{array}{cc|cccc}
	0 & 0 & 0 & 0 & 0 & *\\
	* & 0 & 1 & 0 & 1 & *\\
	\hline
	-1 & 0 & -2 & 0 & 0 & *\\
	-1 & 0 & 0 & -1 & -1 & *\\
	0 & 0 & 0 & 1 & -1 & *\\
	* & 0 & 0 & 0 & 0 & *\\
	\vdots & \vdots & \vdots & \vdots & \vdots & \vdots\\
	* & 0 & 0 & 0 & 0 & *
	\end{array}
	\right),
	\]
    and the index set $I$ should satisfy the condition
	\[
	\sum_{i\in I}f_{i,g} \equiv 2\xi_g \equiv 0 \mod 2.
	\]
    Consider the entry of index 0, by a similar argument in case (1), we have $1,2,3 \in I$ and
	\[
	f_{1,g}+f_{2,g}+f_{3,g}+\sum_{i \in I-\{1,2,3\}}f_{i,g}=2[0,*,-1,-1,0,*\dots,*]^T=2\xi_g.
	\]
	
	By comparing the entries of index 1 of $\xi_g$, we have $\xi_g \notin \bZ\langle f_{1,g},\dots,f_{n,g}\rangle$. Consequently, we have
    $$
    \frac{1}{2}\sum_{i \in I}f_i=\xi \notin \bZ \langle f_{1},\dots,f_{n}\rangle,
    $$
    which gives a nonzero element in    
    $Z^1(G,\Pic(\bar{X}))/\bZ\langle f_{1},\dots,f_{n}\rangle$.

    Similar arguments apply to the case where $2,3\in I$.
	\item Case (3):
	
	If a generator $g=\beta_1\beta$ with $$
    \beta_1=c_{j_1}\dots c_{j_{2h}}(1\ \dots \ w),
    $$ and let $s$ be the sign function of $\beta_1$.
 
	We focus on signed permutation cycle $\beta_1$ and
	\begin{align*}
	g-\mathbb{I}=
	\left(
	\small
	\begin{array}{cc|ccccc}
	0 & 0 & 0 & 0 & \dots & 0 & *\\
	* & 0 & \frac{1-s(w)}{2} & \frac{1-s(1)}{2} & \dots & \frac{1-s(w-1)}{2} & *\\
	\hline
	\frac{s(1)-1}{2} & 0 & -1 & s(1) & \dots & 0 & *\\
	\frac{s(2)-1}{2} & 0 & 0 & -1 & \dots & 0 & *\\
	\vdots & \vdots & \vdots & \vdots & \ddots & \vdots & \vdots\\
	\frac{s(w)-1}{2} & 0 & s(w) & 0 & \dots & -1 & *\\
	* & * & * & * & * & * & *
	\end{array}
	\right).
	\end{align*}
 
	\begin{itemize}
		\item 	If $h=0$, meaning there are no minus signs in the signed permutation cycle, we need to solve the linear equation
		\[
		\frac{1}{2}\sum_{i=1}^wf_{i,g}=0=\sum_{j=1}^w\alpha_{j,I,g}f_{j,g}.
		\]
		This leads us to conclude that $\alpha_{1,I,g}=\dots=\alpha_{w,I,g}$.
		
        \item	If $h \geq 1$, then we have
		\[
		\frac{1}{2}\sum_{i=1}^wf_{i,g}=\sum_{u=1}^h\sum_{r=(j_{2u-1})+1}^{j_{2u}}f_{r,g}
		\]
		where we identify $f_{r}=f_{r-w}$ if $r>w$ and let $j_{2h+1}=j_{1}+w$.
		
		To find other solutions, we solve the following linear equation:
		\[
		0=\sum_{j=1}^w\alpha_{j,I,g}f_{j,g},
		\]
		The solution $\bZ$-module has rank 1, generated by
		\[
		(\sum_{u=1}^h\sum_{r=(j_{2u-1})+1}^{j_{2u}}f_{r,g})-(\sum_{\text{rest r in $\{1,\dots,w\}$}}f_{r,g}).
		\]
	\end{itemize}
\end{itemize}

In the rest of this section, we compute several examples,
\begin{example}
Let $G=\mathfrak{C}_2 \times \mathfrak{D}_4 \subset W(\mathsf{D}_6)$ be given by 
\[ 
g_1=c_1c_2(1\ 2)(3\ 4), \quad g_2=c_1c_3, \quad 	g_3=c_5c_6.
\]
Under the action of $pr(G)$ on index set $\{1,2,3,4,5,6\}$, there are four orbits, 
$$
\{1,2\},\{3,4\},\{5\},\{6\}.
$$

From the discussion above, generators $g_2$ and $g_3$ indicate that we only need to consider two index sets,
$$
I=\{1,2,3,4\} \text{ or } \{5,6\}
$$
that potentially induce $0 \neq \xi \in Z^1(G,\Pic(\bar{X}))/\bZ\langle f_1,\dots,f_n\rangle$.
\begin{itemize}
	\item For generator $g_1$:
	\begin{itemize}
		
		\item Index set $\{1,2,3,4\}$ corresponds to two signed permutation cycles with even number of minus signs, $c_1c_2(1\ 2)$ and $(3\ 4)$, then
		\begin{align*}
		\frac{1}{2}\sum_{i=1}^4f_{i,g_1}=f_{1,g_1}+&Q_{1,\{1,2\},g_1}(f_{1,g_1}-f_{2,g_1})+Q_{1,\{3,4\},g_1}(f_{3,g_1}+f_{4,g_1})\\
        +&Q_{1,\{5\},g_1}f_{5,g_1}+Q_{1,\{6\},g_1}f_{6,g_1}
		\end{align*}
		
		\item Index set $\{5,6\}$ corresponds to two identity cycles, $(5)$ and $(6)$, then
		\begin{align*}
		\frac{1}{2}\sum_{i=5}^6f_{i,g_1}=0=&Q_{2,\{1,2\},g_1}(f_{1,g_1}-f_{2,g_1})+Q_{2,\{3,4\},g_1}(f_{3,g_t}+f_{4,g_t})\\
        +&Q_{2,\{5\},g_1}f_{5,g_1}+Q_{2,\{6\},g_1}f_{6,g_1}
		\end{align*}
	\end{itemize}
	\item For generator $g_2$:
	\begin{itemize} 
		
		\item Index set $\{1,2,3,4\}$ corresponds to two signed permutation cycles with 1 minus sign $c_1,c_3$ and two identity cycles, then
		\begin{align*}
		\frac{1}{2}\sum_{i=1}^4f_{i,g_2} \notin \bZ\langle f_{1,g_2},\dots,f_{6,g_2}\rangle 
		\end{align*}
		The element 
		\begin{align*}
		\frac{1}{2}\sum_{i=1}^4f_i
		\end{align*}
        generates a nonzero element in $Z^1(G,\Pic(\bar{X}))/\bZ\langle f_1,\dots,f_6\rangle$.
		\item Index set $\{5,6\}$ corresponds to two identity cycles, then
		\[
		\frac{1}{2}\sum_{i=5}^6f_{i,g_2}=0=Q_{1,\{2\},g_2}f_{2,g_2}+Q_{1,\{4\},g_2}f_{4,g_2}+Q_{1,\{5\},g_2}f_{5,g_2}+Q_{1,\{6\},g_2}f_{6,g_2}
		\]
	\end{itemize}
	\item For generator $g_3$:
	\begin{itemize} 
		\item Index set $\{1,2,3,4\}$ corresponds to four identity cycles, then
		\[
		\frac{1}{2}\sum_{i=1}^4f_{i,g_3}=0=Q_{1,\{1\},g_3}f_{1,g_3}+Q_{1,\{2\},g_3}f_{2,g_3}+Q_{1,\{3\},g_3}f_{3,g_3}+Q_{1,\{4\},g_3}f_{4,g_3}
		\]
		\item Index set $\{5,6\}$ corresponds to two signed permutation cycles with 1 minus sign, $c_5,c_6$, then
		\[
		\frac{1}{2}\sum_{i=5}^6f_{i,g_3} \notin \bZ\langle f_{1,g_3},\dots,f_{6,g_3}\rangle,
		\]
		the element
		\[
		\frac{1}{2}\sum_{i=5}^6f_i
		\] generates a nonzero element in $Z^1(G,\Pic(\bar{X}))/\bZ\langle f_1,\dots,f_6\rangle$.
	\end{itemize}
	Both index sets $\{1,2,3,4\}$ and $\{5,6\}$ generate a copy of $\mathbb{Z}/2$ and
    $$
    Z^1(G,\Pic(\bar{X}))/\bZ\langle f_1,\dots,f_n\rangle=(\bZ/2)^2.
    $$
	Since generator $g_2,g_3$ have signed permutation cycles with 1 minus sign, we have 
    $$
    f_{-1} \notin \bZ\langle f_1,\dots,f_6\rangle,
    $$
    and
    $$
    \bZ \langle f_{-1},f_1,\dots,f_n\rangle/\bZ \langle f_1,\dots,f_n\rangle =\bZ/2.
    $$
    We have
	$$
    \mathrm{H}^1(G,\Pic(\bar{X}))=\mathbb{Z}/2.
    $$
\end{itemize}	
\end{example}
\begin{example}
Let $G=\mathfrak{C}_2^2.\mathfrak{D}_4 \subset W(\mathsf{D}_6)$ be given by
\[
g_1=c_1c_2c_3c_4c_5c_6(1\ 2)(3\ 4)(5\ 6), \quad 
g_2=c_1c_2(1\ 2\ 3\ 4). 
\]

Under the action of $pr(G)$ on index set $\{1,2,3,4,5,6\}$, there are two orbits,
$$
\{1,2,3,4\} \text{ and }\{5,6\}.
$$
\begin{itemize}
	\item For generator $g_1$:
	\begin{itemize}
		\item Index set $\{1,2,3,4\}$ corresponds to two signed permutation cycles with even number of minus signs, $c_1c_2(1\ 2)$ and $c_3c_4(3\ 4)$, then
		\begin{align*}
		\frac{1}{2}\sum_{i=1}^4f_{i,g_1}=&f_{2,g_1}+f_{4,g_1}\\
  +&Q_{1,\{1,2\},g_1}(f_{1,g_1}-f_{2,g_1})+Q_{1,\{3,4\},g_1}(f_{3,g_1}-f_{4,g_1})+Q_{1,\{5,6\},g_1}(f_{5,g_1}-f_{6,g_1}).
		\end{align*}
		\item Index set $\{5,6\}$ corresponds to one signed permutation cycle with even number of minus signs, $c_5c_6(5\ 6)$, then
		\begin{align*}
		\frac{1}{2}\sum_{i=5}^6f_{i,g_1}&=f_{5,g_1}\\
        &+Q_{2,\{1,2\},g_1}(f_{1,g_1}-f_{2,g_1})+Q_{2,\{3,4\},g_1}(f_{3,g_1}-f_{4,g_1})+Q_{2,\{5,6\},g_1}(f_{5,g_1}-f_{6,g_1}).
		\end{align*}
	\end{itemize}
	
	\item For generator $g_2$:
	\begin{itemize}
		\item Index set $\{1,2,3,4\}$ corresponds to one signed permutation cycle with even number of minus signs, $c_1c_2(1\ 2\ 3\ 4)$, then
		\begin{align*}
		\frac{1}{2}\sum_{i=1}^4f_{i,g_2}&=f_{2,g_2}\\
        &+Q_{1,\{1,2,3,4\},g_2}(f_{2,g_2}-f_{3,g_2}-f_{4,g_2}-f_{1,g_2})+Q_{1,\{5\},g_2}f_{5,g_2}+Q_{1,\{6\},g_2}f_{6,g_2}.
		\end{align*}
		\item Index set $\{5,6\}$ corresponds to two identity cycles, then
		\begin{align*}
		\frac{1}{2}\sum_{i=5}^6f_{i,g_2}&=0\\
        &=Q_{2,\{1,2,3,4\},g_2}(f_{2,g_2}-f_{3,g_2}-f_{4,g_2}-f_{1,g_2})+Q_{2,\{5\},g_2}f_{5,g_2}+Q_{2,\{6\},g_2}f_{6,g_2}.
		\end{align*}
	\end{itemize}
	
	To find $\cap_{i=1}^2V_{I,g_i}$ with index set $I=\{1,2,3,4\}$, we need to solve the following linear system:
    \begin{align*}
        &\alpha_{1,I}:Q_{1,\{1,2\},g_1}=-Q_{1,\{1,2,3,4\},g_2},\\
        &\alpha_{2,I}:1-Q_{1,\{1,2\},g_1}=1+Q_{1,\{1,2,3,4\},g_2},\\
        &\alpha_{3,I}:Q_{1,\{3,4\},g_1}=-Q_{1,\{1,2,3,4\},g_2},\\
        &\alpha_{4,I}:1-Q_{1,\{3,4\},g_1}=-Q_{1,\{1,2,3,4\},g_2},\\
        &\alpha_{5,I}:Q_{1,\{5,6\},g_1}=Q_{1,\{5\},g_2},\\
        &\alpha_{6,I}:-Q_{1,\{5,6\},g_1}=Q_{1,\{6\},g_2}.\\
    \end{align*}
     From equations of $\alpha_{3,I},\alpha_{4,I}$, we have $Q_{1,\{1,2,3,4\},g_2}=-\frac{1}{2}$, that implies there is no integer solution of this linear system, thus $\frac{1}{2}\sum_{i=1}^4f_i$ generates a copy of $\mathbb{Z}/2$ in $Z^1(G,\Pic(\bar{X}))/\bZ\langle f_1,\dots,f_6\rangle$.
	
	To find $\cap_{i=1}^2 V_{I,g_i}$ with index set $I=\{5,6\}$, we need to solve the following linear system:
    \begin{align*}
        &\alpha_{1,I}:Q_{2,\{1,2\},g_1}=-Q_{2,\{1,2,3,4\},g_2},\\
        &\alpha_{2,I}:-Q_{2,\{1,2\},g_1}=Q_{2,\{1,2,3,4\},g_2},\\
        &\alpha_{3,I}:Q_{2,\{3,4\},g_1}=-Q_{2,\{1,2,3,4\},g_2},\\
        &\alpha_{4,I}:-Q_{2,\{3,4\},g_1}=-Q_{2,\{1,2,3,4\},g_2},\\
        &\alpha_{5,I}:1+Q_{2,\{5,6\},g_1}=Q_{2,\{5\},g_2},\\
        &\alpha_{6,I}:-Q_{2,\{5,6\},g_1}=Q_{2,\{6\},g_2}.\\
    \end{align*} 
    There is an integer solution given by
	\begin{align*}
	Q_{2,\{1,2\},g_1}=Q_{2,\{3,4\},g_1}=Q_{2,\{5,6\},g_1}=0,\\
	Q_{2,\{1,2,3,4\},g_2}=Q_{2,\{6\},g_2}=0,\ Q_{2,\{5\},g_2}=1,
	\end{align*}
	thus $\frac{1}{2}\sum_{i=5}^6 f_i$ does not generate a copy of $\mathbb{Z}/2$ in $Z^1(G,\Pic(\bar{X}))/\bZ\langle f_1,\dots,f_6\rangle$.
	
	It follows that
    \begin{align*}
    Z^1(G,\Pic(\bar{X}))/\bZ\langle f_1,\dots,f_6\rangle=\bZ/2.
	\end{align*}
    
	Next, we compute
    \begin{align*}
         \bZ\langle f_{-1},f_1,\dots,f_6\rangle/\bZ\langle f_1,\dots,f_6\rangle,
    \end{align*}
    which requires solving the following linear system:
	\begin{align*}
	f_{-1,g_1}&=\frac{1}{2}\sum_{i=1}^6f_{i,g_1}=f_{2,g_1}+f_{4,g_1}+f_{6,g_1}+\\
	&\bar{Q}_{1,\{1,2\},g_1}(f_{1,g_1}-f_{2,g_1})+\bar{Q}_{1,\{3,4\},g_1}(f_{3,g_1}-f_{4,g_1})+\bar{Q}_{1,\{5,6\},g_1}(f_{5,g_1}-f_{6,g_1}),\\
    \end{align*}
    and
    \begin{align*}
	f_{-1,g_2}&=\frac{1}{2}\sum_{i=1}^6f_{i,g_2}=f_{2,g_2}\\
    +&\bar{Q}_{2,\{1,2,3,4\},g_2}(f_{2,g_2}-f_{3,g_2}-f_{4,g_2}-f_{1,g_2})+\bar{Q}_{2,\{5\},g_2}f_{5,g_2}+\bar{Q}_{2,\{6\},g_2}f_{6,g_2}.
	\end{align*}
	We have $\cap_{i=1}^2V_{I,g_i}=\emptyset$ with $I=\{-1\}$, thus $f_{-1} \notin \bZ\langle f_1,\dots,f_6\rangle$. Therefore, 
    $$
    \bZ\langle f_{-1},f_1,\dots,f_6\rangle/\bZ\langle f_1,\dots,f_6\rangle=\bZ/2,
    $$
    and $\mathrm{H}^1(G,\Pic(\bar{X}))=0$.
\end{itemize}
\end{example}

\begin{example}
Consider a standard conic bundle $X$ over $k=\mathbb{Q}$ with an affine patch,
\begin{align*}
    X_1:\{y^2-5z^2=x^5-2\} \subset \mathbb{A}_k^3(x,y,z).
\end{align*}
and let
\begin{align*}
    P(x)=x^5-2
\end{align*}
The splitting field is $K'=\mathbb{Q}(2^{\frac{1}{5}},\epsilon_5=e^{2\pi i/5})$ and the Galois group $G=\Gal(K'/k)=\mathfrak{F}_5$ which is case 2 in section \ref{sect:classes} with $p=5,r=1$.

There are 6 pairs of degenerate fibers:
\begin{align*}
    &q_i^{\pm}:\{ y=\pm \sqrt{5}z,x=2^{\frac{1}{5}}\epsilon_5^{i-1}\},\ i=1,\dots,5,\\
    &q_6^{\pm}: \text{degenerate fibers over $\infty$}.
\end{align*}

The Galois group $G$ is generated by
\begin{align*}
    &\sigma:2^{\frac{1}{5}} \to 2^{\frac{1}{5}},\quad \epsilon_5 \to \epsilon_5^2,\\
    &\omega:2^{\frac{1}{5}} \to 2^{\frac{1}{5}}\epsilon_5,\quad \epsilon_5 \to \epsilon_5.
\end{align*}

Under the action of $\sigma$,
\begin{align*}
    \sqrt{5}=2(\epsilon_5+\epsilon_5^4)+1 \xrightarrow{\sigma} 2(\epsilon_5^2+\epsilon_5^3)+1=-\sqrt{5}.
\end{align*}

On degenerate fibers, the action of $G$ is given by
\begin{align*}
\sigma:q_1^{\pm} \to q_1^{\mp}, q_2^{\pm} \to q_3^{\mp}, q_3^{\pm} \to q_5^{\mp}, q_4^{\pm} \to q_2^{\mp}, q_5^{\pm} \to q_4^{\mp}, q_6^{\pm} \to q_6^{\mp},\\
\omega:q_1^{\pm} \to q_2^{\pm}, q_2^{\pm} \to q_3^{\pm}, q_3^{\pm} \to q_4^{\pm}, q_4^{\pm} \to q_5^{\pm}, q_5^{\pm} \to q_1^{\pm}, q_6^{\pm} \to q_6^{\pm}.
\end{align*}

In terms of signed permutation representation, the group $G$ is written as
\begin{align*}
    &\sigma=c_1c_2c_3c_4c_5c_6(2,3,5,4),\\
    &\omega=(1,2,3,4,5).
\end{align*}

Under the action of $pr(G)$ on index set $\{1,2,3,4,5,6\}$, there are two orbits,
\begin{align*}
    \{1,2,3,4,5\} \text{ and } \{6\}
\end{align*}

Generator $\sigma$ is discussed in case (1) above, we  have $1,6 \in I$, so we only consider index set
\begin{align*}
    I=\{1,2,3,4,5,6\},
\end{align*}
and
\begin{align*}
    \frac{1}{2}\sum_{i\in I}f_i=f_{-1} \notin \bZ\langle f_1,\dots,f_6\rangle.
\end{align*}

We have
\begin{align*}
    &Z^1(G,\Pic(\bar{X}))/\bZ\langle f_1,\dots,f_n\rangle=\bZ/2,\\
    &\bZ \langle f_{-1},f_1,\dots,f_n\rangle/\bZ\langle f_1,\dots,f_n\rangle=\bZ/2,
\end{align*}
and
\begin{align*}
    \mathrm{H}^1(G,\Pic(\bar{X}))=0.
\end{align*}

\end{example}

\section{Projection with respect to orbit decomposition}
\label{sect:proj}
Let $G \subset W(\mathsf{D}_n)$ be a group satisfying the (H1), relative minimality and minimality conditions.
From \cite{kst}, a relatively $k$-minimal but not $k$-minimal conic bundle can only occur if its degree is equal to 3,5,6,7 or 8.
When the degree of the conic bundle is less than 2, the minimality condition and relative minimality condition coincide.

Since $G$ is minimal, the group $G$ gives a orbit decomposition on set $\{1,\dots,n\}$ and pick an orbit $O$ of length $n'$, without loss of generality, can assume $O=\{1,\dots,n'\}$.

For every group element $g=\beta_{(g,1)}\dots\beta_{(g, R_g)}$ where all $\beta$ are disjoint signed permutation cycles, we assume that the first $r_g$ signed permutation cycles describe the group action of element $g$ on orbit $O$.
 
We construct a set $\overline{G}_O$ defined by
$$
\overline{G}_O=\{P_O(g)|g \in G\},
$$
where
\[
P_O(g)=\overline{g}=
\begin{cases}
\beta_{(g,1)}\dots\beta_{(g,r_g)}c_{n'+1},\quad &\text{if $\sigma(\beta_{(g,1)}\dots\beta_{(g,r_g)})=-1$},\\
\beta_{(g,1)}\dots\beta_{(g,r_g)},\quad &\text{if $\sigma(\beta_{(g,1)}\dots\beta_{(g,r_g)})=1$}.
\end{cases}
\]

\begin{lemma}{\label{lemma 14}}
The set $\overline{G}_O$ is a group and $P_O$ is a group homomorphism
\end{lemma}
\begin{proof}
	We prove $\overline{G}_O$ is a group by definition.
	\begin{itemize}
		\item Identity:
		
			Since $\mathbb{I} \in G$, then $P_O(\mathbb{I})=\overline{\mathbb{I}} \in \overline{G}_O$.
		\item closure:
		
			If $\overline{g},\overline{h} \in \overline{G}_O$, then there are elements $g,h \in G$ such that $P_O(g)=\overline{g}$ and $P_O(h)=\overline{h}$. 
            
            Assume they have the form
			\begin{align*}
			&g=\beta_{(g,1)}\dots\beta_{(g,r_g)}\beta_{(g,r_g+1)}\dots\beta_{(g,R_g)},\\
			&h=\beta_{(h,1)}\dots\beta_{(h,r_h)}\beta_{(h,r_h+1)}\dots\beta_{(h,R_h)},\\
			&\overline{g}=\beta_{(g,1)}\dots\beta_{(g,r_g)},\\
			&\overline{h}=\beta_{(h,1)}\dots\beta_{(h,r_h)},
			\end{align*}
			then we have
			\begin{align*}
			gh=(\beta_{(g,1)}\dots\beta_{(g,r_g)})(\beta_{(g,r_g+1)}\dots\beta_{(g,R_g)})(\beta_{(h,1)}\dots\beta_{(h,r_h)})(\beta_{(h,r_h+1)}\dots\beta_{(h,R_h)})\\
			\end{align*}
			Since $\beta_{(g,r_g+1)}\dots\beta_{(g,R_g)}$ and $\beta_{(h,1)}\dots\beta_{(h,r_h)}$ describe group action on different orbits, so they commute and
			
			\begin{align*}
			gh=(\beta_{(g,1)}\dots\beta_{(g,r_g)})(\beta_{(h,1)}\dots\beta_{(h,r_h)})(\beta_{(g,r_g+1)}\dots\beta_{(g,R_g)})(\beta_{(h,r_h+1)}\dots\beta_{(h,R_h)}),
			\end{align*}
			then $P_O(gh)=\overline{gh}=\overline{g}\overline{h}=P_O(g)P_O(h)$ and we have $\overline{g}\overline{h} \in \overline{G}_O$.
            
            For the other two cases with $$
            \overline{g}=\beta_{(g,1)}\dots\beta_{(g,r_g)}c_{n'+1},
            $$
            or 
            $$
            \overline{h}=\beta_{(h,1)}\dots\beta_{(h,r_h)}c_{n'+1},
            $$
            the arguments are similar.
		\item associative:
		
			If $\overline{g},\overline{g'},\overline{g''} \in \overline{G}_O$ and $P_O(g)=\overline{g}, P_O(g')=\overline{g'}$ and $P_O(g'')=\overline{g''}$.
			
			Using the proof in the previous point, we have
			\begin{align*}
			(\overline{g}\overline{g'})\overline{g''}=(\overline{gg'}) \overline{g''}=P_O((gg')g'')=P_O(g(g'g''))=\overline{g}(\overline{g'g''})=\overline{g}(\overline{g'}\ \overline{g''})
			\end{align*}
		\item Inverse:
		
			Let $\overline{g} \in \overline{G}_O$ and $P_O(g)=\overline{g}$.
			
			In $G$, we have the inverse of $g$ exists and $gg^{-1}=\mathbb{I}$, then $P_O(gg^{-1})=P_O(\mathbb{I})$.
			We have
			\[
			\overline{g}\overline{g^{-1}}=\overline{\mathbb{I}}
			\]
			and
			\[
			\overline{g}^{-1}=\overline{g^{-1}} \in \overline{G}_O 
			\]
	\end{itemize}
 
	In the proof of the closure property of $\overline{G}_O$, we already have
	\[
	P_O(gh)=\overline{gh}=\overline{g}\overline{h}=P_O(g)P_O(h) \quad \forall g,h \in G,
	\]
	
	and it shows that $P_O$ is a group homomorphism.
\end{proof}

Group $\overline{G}_O$ is either a subgroup of $W(\mathsf{D}_{n'})$ with transitive orbit or a subgroup of $W(\mathsf{D}_{n'+1})$ with two orbits of length $n'$ and $1$.

From Lemma \ref{lemma 14}, given an orbit $O$, there is a map
\begin{align*}
P_{O}:\{H:H\leq G\} \to \{\overline{H}_O:\overline{H}_O\leq \overline{G}_O\}
\end{align*}
given by
\[
P_O(H)=\overline{H}_O=\{\overline{h}=P_O(h)| h \in H\}
\]

\begin{lemma}
The group $\overline{G}_O$ satisfies both relative minimality and $\mathrm{(H1)}$ conditions.
\end{lemma}
\begin{proof}
We divide the proof into two cases:
\begin{enumerate}
\item If $\overline{G}_O \subset W(\mathsf{D}_{n'})$.
	\begin{itemize}
		\item Relative minimality condition of $\overline{G}_O$.
		
		      If $\overline{G}_O$ is not relatively minimal, then 
			\[
			(\mathbb{Z}^{n'+2})^{\overline{G}_O}\neq (\mathbb{Z}\overline{l_0}\oplus\mathbb{Z}\overline{K}) 
			\]
			where
			\[
			\overline{l_0}=[0,1,0,\dots,0]^T
			\]
                and
			\[
			\overline{K}=[-2,-2,1,\dots,1]^T
			\] having length $n'+2$.

			Therefore, we can find an element $\overline{v} \in (\mathbb{Z}^{n'+2})^{\overline{G}_O}\setminus (\mathbb{Z}\overline{l_0}\oplus\mathbb{Z}\overline{K})$ of the form
			\begin{align*}
			\overline{v}=[x_{-1},0,x_1,\dots,x_{n'}]^T,
			\end{align*}
			such that
			\begin{align*}
			2\overline{v}+x_{-1}(\overline{K}+2\overline{l_0})=[0,0,2x_1+x_{-1},\dots,2x_{n'}+x_{-1}]^T=\overline{\zeta}.
			\end{align*}
   
			The element $\overline{\zeta} \neq 0$, otherwise $\overline{v}=-\frac{x_{-1}}{2}(\overline{K}+2\overline{l_0})$. In this setting, either $\overline{v}$ is in $\mathbb{Z}\overline{l_0}\oplus\mathbb{Z}\overline{K}$ or it contains half-integers entries. Both cases lead to a contradiction.
			
			Extend $\overline{\zeta}$ by $0$ to produce an element $\zeta\in\Pic(\bar{X})=\mathbb{Z}^{n+2}$ where
			\[
			\zeta=[0,0,2x_1+x_{-1},\dots,2x_{n'}+x_{-1},0,\dots,0]^T,	
			\]
			and it is invariant under the action of $G$, i.e., $\zeta \in (\Pic(\bar{X}))^G$.
            
			Obviously $\zeta \notin \mathbb{Z}l_0\oplus\mathbb{Z}K_X$, where
			\[
			l_0=[0,1,0,\dots,0]^T
			\]and
			\[
			K_X=[-2,-2,1,\dots,1]^T
			\]
			having length $n+2$.
   
            Therefore,
            $$
            \Pic(\bar{X})^G \neq \mathbb{Z}l_0\oplus\mathbb{Z}K_X,
            $$
            and $G$ is not relatively minimal, which contradicts our assumption.
		\item (H1) condition of $\overline{G}_O$.
		
			Let $\overline{H}_O$ be a subgroup of $\overline{G}_O$ and let $H=P_O^{-1}(\overline{H}_O)$ and assume 
            \begin{align*}
            \mathrm{H}^1(\overline{H}_O,\mathbb{Z}^{n'+2})=(\mathbb{Z}/2)^{\overline{\Lambda}} \text{ and } \rH^1(H,\Pic(\bar{X}))=(\mathbb{Z}/2)^{\Lambda}.
            \end{align*}
			
			We treat $H$( or $\overline{H}_O$) as a set and use this set as the set of generators of $H$( or $\overline{H}_O$).
			
			By the same argument in Theorem \ref{thm:13}, there is a group element of $H$( or $\overline{H}_O$) having a signed permutation cycle with one minus sign, then we have
			\begin{align*}
			f_{-1} \notin \bZ\langle f_1,\dots,f_n\rangle,
			\end{align*}
			and
			\begin{align*}
			\overline{f}_{-1} \notin \bZ\langle \overline{f}_1,\dots,\overline{f}_{n'}\rangle,
			\end{align*}
			where $\overline{f}_i$ are columns of $B^1(\overline{H}_O,\bZ^{n'+2})$.
   
			Furthermore, we have
			\begin{align*}
			&Z^1(H,\Pic(\bar{X}))/\bZ\langle f_1,\dots,f_n\rangle=(\mathbb{Z}/2)^{\Lambda+1},\\
			&Z^1(\overline{H}_O,\mathbb{Z}^{n'+2})/\bZ\langle \overline{f}_1,\dots,\overline{f}_{n'}\rangle=(\mathbb{Z}/2)^{\overline{\Lambda}+1}.
			\end{align*}
		
            If there is a subset $I \subset \{1,\dots,n' \}$ such that
		\begin{align*}
		\frac{1}{2}\sum_{i \in I}\overline{f}_i \in Z^1(\overline{H}_O,\mathbb{Z}^{n'+2})\setminus \bZ\langle \overline{f}_1,\dots,\overline{f}_{n'}\rangle,
			\end{align*}
			where $\overline{f}_i$ are columns of $B^1(\overline{H}_O,\mathbb{Z}^{n'+2})$, we conclude:
			\begin{enumerate}
				\item For every element $\overline{h}\in\overline{H}_O$, we have
				\begin{align*}
				\sum_{i \in I}\overline{f}_{i,\overline{h}} = 0 \mod 2.
				\end{align*}
				\item There is no integer solution $\{\alpha_i\}$ such that
				\begin{align*}
				\frac{1}{2}\sum_{i \in I}\overline{f}_{i,\overline{h}}=\sum_{i=1}^{n'}\alpha_i\overline{f}_{i,\overline{h}} \in \bZ\langle \overline{f}_{1,\overline{h}},\dots,\overline{f}_{n',\overline{h}}\rangle,
				\end{align*}
				for all $\overline{h} \in \overline{H}_O$.
			\end{enumerate}
			We observe that
			\begin{align*}
			(\overline{f}_{i,\overline{h}})_j=(f_{i,h})_j,\quad j\in \{-1,\dots,n'\}
			\end{align*}
			for all pairs of $h,\overline{h}$ satisfy $P_O(h)=\overline{h}$. Here, $f_{i,h}$ extends $\overline{f}_{i,\overline{h}}$ by adding zeros. Consequently, we have
			\[
			\sum_{i \in I}f_{i,h} = 0 \mod 2 \quad \forall h\in H,
			\]
			and there is no integer solution $\{\alpha_i\}$ such that
			\begin{align*}
			\frac{1}{2}\sum_{i \in I}f_{i,h}=\sum_{i=1}^{n'}\alpha_if_{i,h} \in \bZ\langle f_{1,h},\dots,f_{n',h}\rangle
			\end{align*}
			for all $h \in H$.
			
			Therefore, we have
			\begin{align*}
			\frac{1}{2}\sum_{i \in I}f_i \notin \bZ\langle f_1,\dots,f_{n'}\rangle.
			\end{align*}
			It is obvious that
                $$
                \frac{1}{2}\sum_{i \in I}f_i \notin \bZ\langle f_{n'+1},\dots,f_n\rangle,
                $$
                and we have
			\[
			\frac{1}{2}\sum_{i \in I}f_i\notin \bZ\langle f_1,\dots,f_n\rangle.
			\]
			
			We define a map
            \begin{align*}
                \kappa:Z^1(\overline{H}_O,\mathbb{Z}^{n'+2}) \to Z^1(H,\Pic(\bar{X}))
            \end{align*}
            by
            \begin{align*}
            \kappa(\frac{1}{2}\sum_{i\in I}\overline{f}_i)=\frac{1}{2}\sum_{i\in I}f_i,
            \end{align*}
			and it is obvious
            \begin{align*}
            \kappa(\overline{f}_i)=f_i,\quad i=1,\dots,n'.
            \end{align*}
            With abuse of notation, define the map
            \begin{align*}
            \kappa:Z^1(\overline{H}_O,\mathbb{Z}^{n'+2})/\bZ\langle \overline{f}_1,\dots,\overline{f}_{n'}\rangle \to Z^1(H,\Pic(\bar{X}))/\bZ \langle f_1,\dots,f_n\rangle,
            \end{align*}
			and we show this map $\kappa$ is injective.
   
            If there are two subsets $I\ne I' \subset \{1,\dots,n'\}$ such that
            
			\begin{align*}
			\frac{1}{2}(\sum_{i \in I}f_i-\sum_{i \in I'}f_i) \in \bZ\langle f_1,\dots,f_n\rangle,
			\end{align*}
			then
			\begin{align*}
			\frac{1}{2}(\sum_{i \in I}f_i-\sum_{i \in I'}f_i)=\sum_{i=1}^{n'}\alpha_if_i \in \bZ \langle f_1,\dots,f_{n'}\rangle,   
			\end{align*}
			and it implies
			\begin{align*}
			\frac{1}{2}(\sum_{i \in I}f_{i,h}-\sum_{i \in I'}f_{i,h})=\sum_{i=1}^{n'}\alpha_if_{i,h}, \quad \forall h \in H.
			\end{align*}
			By deleting 0 entries corresponding to other orbits from columns $f_{1,h},\dots,f_{n',h}$, we have
			\begin{align*}
			\frac{1}{2}(\sum_{i \in I}\overline{f}_{i,\overline{h}}-\sum_{i \in I'}\overline{f}_{i,\overline{h}})=\sum_{i=1}^{n'}\alpha_i\overline{f}_{i,\overline{h}}, \quad \forall \overline{h} \in \overline{H}_O,
			\end{align*}
			which implies
			\begin{align*}
			\frac{1}{2}(\sum_{i \in I}\overline{f}_i-\sum_{i \in I'}\overline{f}_i) \in \bZ \langle \overline{f}_1,\dots,\overline{f}_{n'}\rangle.
			\end{align*}
			We show this map is injective. Consequently, we have $\Lambda+1 \geq \overline{\Lambda}+1 \Longleftrightarrow \Lambda \geq \overline{\Lambda}$.
			
			Since $G$ satisfies (H1) condition, so does $\overline{G}_O$.
	\end{itemize}
\item If $\overline{G}_O \subset W(\mathsf{D}_{n'+1})$
	\begin{itemize}
		\item Relative minimality condition of $\overline{G}_O$.
            
			If $\overline{G}_O$ is not relative minimal, then
   
			\begin{align*}
			(\mathbb{Z}^{n'+3})^{\overline{G}_O} \neq \mathbb{Z}\overline{l_0}+\mathbb{Z}\overline{K} ,
			\end{align*}
			where
			\begin{align*}
			\overline{l_0}=[0,1,0,\dots,0]^T,
			\end{align*}
            and
			\begin{align*}
			\overline{K}=[-2,-2,1,\dots,1]^T,
			\end{align*}
            having length $n'+3$.
			
			We can find an element $\overline{v} \in (\mathbb{Z}^{n'+3})^{\overline{G}_O}\setminus (\mathbb{Z}\overline{l_0}+\mathbb{Z}\overline{K})$ of the form
			\[
			\overline{v}=[x_{-1},0,x_1,\dots,x_{n'},x_{n'+1}]^T.
			\]
                Since $\overline{G}_O$ is considered as a subgroup of $W(\sf{D}_{n'+1})$, there exists an element $\overline{g}$ of the form $\overline{g}=\beta_1\dots\beta_{r}c_{n'+1}$, then we have
			\[
			\overline{g}-\mathbb{I}=
			\left(
			\begin{array}{cccc}
			0 & 0 & 0 & 0\\
			* & 0 & B_{\Psi(\beta_1\dots\beta_r)} & 1\\
			A_{\Psi(\beta_1\dots\beta_r)} & 0 & \Psi(\beta_1\dots\beta_r)-\mathbb{I} & 0\\
			-1 & 0 & 0 & -2
			\end{array}
			\right),
			\]
            where $\Psi:W(\mathsf{D}_{n'}) \to \GL_{n'}(\mathbb{Z})$ is the representation given in section \ref{sect:conic}.
           
			From the last row of $\overline{g}-\mathbb{I}$, we have relation:
			\[
			x_{-1}+2x_{n'+1}=0,
			\]
			and therefore
			\begin{align*}
			\overline{v}-x_{n'+1}(\overline{K}+2\overline{l_0})
            =[0,0,x_1-x_{n'+1},\dots,x_{n'}-x_{n'+1},0]^T=\overline{\zeta}\neq 0.
			\end{align*}
            
			Extend $\overline{\zeta}$ by zero to produce an element $\zeta$ in $\Pic(\bar{X})=\mathbb{Z}^{n+2}$ where
			\begin{align*}
			\zeta=[0,0,x_1-x_{n'+1},\dots,x_{n'}-x_{n'+1},0,\dots,0]^T,
			\end{align*}
			and it is invariant under the action of $G$, i.e., $\zeta \in (\Pic(\bar{X}))^G$.
            
			Obviously $\zeta \notin \mathbb{Z}l_0\oplus\mathbb{Z}K_X$, where
			\[
			l_0=[0,1,0,\dots,0]^T
			\]and
			\[
			K_X=[-2,-2,1,\dots,1]^T
			\]
			are of length $n+2$.
            
            Thus
            $$
            \Pic(\bar{X})^G \neq \mathbb{Z}l_0\oplus\mathbb{Z}K_X,
            $$
            and $G$ is not relatively minimal, which contradicts our assumption.
		\item (H1) condition of $\overline{G}_O$.
		
			Let $\overline{H}_O$ be a subgroup of $\overline{G}_O$ and let $H=P_O^{-1}(\overline{H}_O)$ and assume \begin{align*}
            \mathrm{H}^1(\overline{H}_O,\mathbb{Z}^{n'+3})=(\mathbb{Z}/2)^{\overline{\Lambda}} \text{ and } \rH^1(H,\Pic(\bar{X}))=(\mathbb{Z}/2)^{\Lambda}.
            \end{align*}
			
			We treat $H$( or $\overline{H}_O$) as a set and use this set as the set of generators of $H$( or $\overline{H}_O$).
		
			By the same argument in Theorem \ref{thm:13}, there is a group element of $H$( or $\overline{H}_O$) having a signed permutation cycle with one minus sign.
		
			Then we have
			\begin{align*}
			f_{-1} \notin \bZ \langle f_1,\dots,f_n\rangle,
			\end{align*}
			and
			\begin{align*}
			\overline{f}_{-1} \notin \bZ\langle \overline{f}_1,\dots,\overline{f}_{n'}\rangle,
			\end{align*}
		  where $\overline{f}_i$ are columns of $B^1(\overline{H}_O,\bZ^{n'+3})$ and
			\begin{align*}
			    &Z^1(H,\Pic(\bar{X}))/\bZ \langle f_1,\dots,f_n\rangle =(\mathbb{Z}/2)^{\Lambda+1},\\
                &Z^1(\overline{H}_O,\mathbb{Z}^{n'+2})/\bZ \langle \overline{f}_1,\dots,\overline{f}_{n'}\rangle=(\mathbb{Z}/2)^{\overline{\Lambda}+1}.
			\end{align*}
			
			To show $\bar{\Lambda} \leq \Lambda$, we divide the proof into two cases:
			\begin{enumerate}
				\item For a subset $I \subset \{1,\dots,n'\}$ such that
				\begin{align*}
				\frac{1}{2}\sum_{i \in I}\overline{f}_i \in Z^1(\overline{H}_O,\mathbb{Z}^{n'+3})\setminus \bZ \langle \overline{f}_1,\dots,\overline{f}_{n'+1}\rangle,
				\end{align*}
				by the same argument in the previous case, we have $\frac{1}{2}\sum_{i \in I}f_i$ generates a nonzero element in $Z^1(H,\Pic(\bar{X}))/\bZ\langle f_1,\dots,f_n\rangle$.
				\item Assume there is a subset $I \subset \{1,\dots,n'\}$ such that
				\begin{align*}
				\frac{1}{2}(\sum_{i \in I}\overline{f}_i+\overline{f}_{n'+1}) \in Z^1(\overline{H}_O,\mathbb{Z}^{n'+3})\setminus \bZ \langle \overline{f}_1,\dots,\overline{f}_{n'+1}\rangle,
				\end{align*}

                then it implies
				\begin{align*}
				\sum_{i \in I}f_{i,h}+\sum_{i=n'+1}^nf_{i,h} \equiv 0 \mod 2,
				\end{align*}
				for all elements $h\in H$ and
				\begin{align*}
				\frac{1}{2}(\sum_{i \in I}f_i+\sum_{i=n'+1}^nf_i) \in Z^1(H,\Pic(\bar{X})).
				\end{align*}
				
	            There exists an element $h \in H$ of the form
                $$
                h=\beta_{(h,1)}\dots\beta_{(h,r_h)}\beta_{(h,r_h+1)}\dots\beta_{(h,R_h)},
                $$
                with $\sigma(\beta_{(h,1)}\dots\beta_{(h,r_h)})=-1$.
                This element corresponds to
                \begin{align*}
                P_O(h)=\overline{h}=\beta_{(h,1)}\dots\beta_{(h,r_h)}c_{n'+1}.
                \end{align*}
                Since $\sigma(\beta_{(h,r_h+1)}\dots\beta_{(h,R_h)})=-1$, it follows that among the signed permutation cycles $\beta_{(h,r_h+1)},\dots,\beta_{(h,R_h)}$, there is a signed permutation cycle with one minus sign, thus we have 
			    $$
                \frac{1}{2}(\sum_{i \in I}f_i+\sum_{i=n'+1}^nf_i) \notin \bZ \langle f_1,\dots,f_n\rangle,
                $$
				which induces a nonzero element in
                $$
                Z^1(H,\Pic(\bar{X}))/\bZ\langle f_1,\dots,f_n \rangle.
                $$
				
				We define a map
                \begin{align*}
                    \kappa:Z^1(\overline{H}_O,\bZ^{n'+3}) \to Z^1(H,\Pic(\bar{X})),
                \end{align*}
                by
                \begin{align*}
                    &\kappa(\frac{1}{2}\sum_{i\in I}\overline{f}_i)=\frac{1}{2}\sum_{i \in I}f_i,\\
                    &\kappa(\frac{1}{2}(\sum_{i\in I}\overline{f}_i+\overline{f}_{n'+1}))=\frac{1}{2}(\sum_{i\in I}f_i+\sum_{i=n'+1}^nf_i).
                \end{align*}
                Obviously, we have
                \begin{align*}
                    &\kappa(\overline{f}_i)=f_i \in \bZ\langle f_1,\dots,f_n\rangle,\quad \forall 1\leq i\leq n',\\
                    &\kappa(\overline{f}_{n'+1})=\sum_{i=n'+1}^nf_i \in \bZ\langle f_1,\dots,f_n\rangle.
                \end{align*}
                With abuse of notation, define the map
				\begin{align*}
				\kappa:Z^1(\overline{H}_O,\mathbb{Z}^{n'+3})/\bZ \langle \overline{f}_1,\dots,\overline{f}_{n'+1}\rangle \to Z^1(H,\Pic(\bar{X}))/\bZ \langle f_1,\dots,f_n\rangle,
				\end{align*}
				and we show this map is injective.
    
				If there are two subsets $I,I'\subset \{1,\dots,n'\}$ such that they give nonzero elements in $Z^1(H,\Pic(\bar{X}))/\bZ\langle f_1,\dots,f_n\rangle$.
    
				There are two cases:
				\begin{enumerate}
					\item If
					\begin{align*}
					&\frac{1}{2}(\sum_{i \in I}f_i+\sum_{i=n'+1}^nf_i)-\frac{1}{2}(\sum_{i \in I'}f_i+\sum_{i=n'+1}^nf_i)\\
					=&\frac{1}{2}(\sum_{i \in I}f_i-\sum_{i \in I'}f_i) \in \bZ\langle f_1,\dots,f_n\rangle
                    \end{align*}
					By the same argument as in the previous case, we have
					\begin{align*}
					&\frac{1}{2}(\sum_{i \in I}\overline{f}_i+\overline{f}_{n'+1})-\frac{1}{2}(\sum_{i \in I'}\overline{f}_i+\overline{f}_{n'+1})\\
					=&\frac{1}{2}(\sum_{i \in I}\overline{f}_i-\sum_{i \in I'}\overline{f}_i) \in \bZ\langle \overline{f}_1,\dots,\overline{f}_{n'}\rangle
					\end{align*}

					\item By a similar argument in case $(b)$, we could find an element $h \in H$ of the form
                    $$
                    h=\beta_{(h,1)}\dots\beta_{(h,r_h)}\beta_{(h,r_h+1)}\dots\beta_{(h,R_h)}, 
                    $$
                    with $\sigma(\beta_{(h,1)}\dots\beta_{(h,r_h)})=-1$
                    and
                    \begin{align*}
                    P_O(h)=\overline{h}=\beta_{(h,1)}\dots\beta_{(h,r_h)}c_{n'+1}.
                    \end{align*}
                    We have $\sigma(\beta_{(h,r_h+1)}\dots\beta_{(h,R_h)})=-1$, thus among signed permutation cycles $\beta_{(h,r_h+1)},\dots,\beta_{(h,R_h)}$, there is a signed permutation cycle with one minus sign and
                    \begin{align*}
                        \frac{1}{2}\sum_{i=n'+1}^nf_{i,h} \notin \bZ \langle f_{n'+1,h},\dots, f_{n,h}\rangle,
                    \end{align*}
                    and as a consequence,
                    \begin{align*}
                        \frac{1}{2}\sum_{i=n'+1}^nf_i \notin \bZ \langle f_{n'+1},\dots,f_n\rangle,
                    \end{align*}
					which implies
                    \begin{align*}
					\frac{1}{2}(\sum_{i \in I}f_i+\sum_{i=n'+1}^nf_i-\sum_{i \in I'}f_i) \notin \bZ \langle f_1,\dots,f_n\rangle. 
					\end{align*}
				\end{enumerate}

				The map $\kappa$ is injective, and we have $\Lambda+1 \geq \bar{\Lambda}+1 \Longleftrightarrow \Lambda \geq \bar{\Lambda}$.
				
				Since $G$ satisfies (H1) condition, so does $\overline{G}_O$.
			\end{enumerate}
	\end{itemize}
\end{enumerate}
\end{proof}

\section{Building blocks and some general classes}
\label{sect:classes}

For those classes with name $*$, we don't find a general description of groups in the class. By Lemma \ref{lem:sylow2}, the Sylow 2-subgroup of $G$ controls the (H1) condition of group $G$.

If the Sylow 2-subgroup of $G$ is contained in some cyclic group, then by the proposition \ref{h1:cyclic}, we could easily check the (H1) condition. If the Sylow 2-subgroup of $G$ is abelian, we could use Proposition \ref{prop:h1bicyc} to check the (H1) condition. By corollary \ref{h1:cyclic} and Proposition \ref{prop:h1bicyc}, we verify the (H1) condition for all classes except classes 11, 18, 19, 22 and 24.

Next, we verify the (H1) condition for classes 11, 18, 22 and 24 below. Their Sylow 2-subgroups are $\mathfrak{D}_4$, and all proper subgroups of $\mathfrak{D}_4$ are either cyclic or bicyclic. From Proposition \ref{h1:cyclic} and their generators, they all have trivial group cohomology. We only need to verify
\begin{align*}
    \mathrm{H}^1(\mathfrak{D}_4,\Pic(\bar{X}))=0.
\end{align*}
with corresponding action of $\mathfrak{D}_4$ on $\Pic(\bar{X})$.

\begin{example}
\label{exam:caseD4}
    We compute 
    \begin{align*}
        \mathrm{H}^1(\mathfrak{D}_4,\Pic(\bar{X})),
    \end{align*}
    with different groups $\mathfrak{D}_4$ described in classes 11, 18, 22 and 24.
    \begin{itemize}
        \item [(11)]: The Sylow 2-subgroup $\mathfrak{D}_4$ is generated by $g_1,g_3$ where
        \begin{align*}
            &g_1=c_1\dots c_{2n+1}c_{2n+2}(2,3)\dots(2n,2n+1)(2n+2,2n+3),\\
            &g_3=(2n+2,2n+3).
        \end{align*}
        Under the action of $pr(\mathfrak{D}_4)$ on set $\{1,\cdots,2n+3\}$, there are $n+2$ orbits:
        \begin{align*}
            O_1=\{1\},\ O_i=\{2i-2,2i-1\},\quad i=2,\cdots,n+2.
        \end{align*}
        Over generator $g_1$, orbits $O_1$ and $O_{n+2}$ have signed permutation cycle with one minus sign, thus we only consider index sets
        \begin{align*}
            I_1=\{1,2n+2,2n+3\},I_i=\{2i-2,2i-1\},\quad i=2,\cdots,n+1,
        \end{align*}
        where $I_1$ is union of $O_1$ and $O_{n+2}$.

        Over generator $g_1$, index set $I_1$ has two signed permutation cycles with one minus sign, $c_1$ and $c_{2n+2}(2n+2,2n+3)$, then
        \begin{align*}
            \frac{1}{2}\sum_{i\in I_1}f_{i}\notin \bZ \langle f_1,\cdots,f_{2n+3}\rangle.
        \end{align*}

        For other index sets $I_j$ with $j=2,\cdots,n+1$, we have
        \begin{align*}
            \frac{1}{2}\sum_{i \in I_j}f_i=f_{2j-1} \in \bZ \langle f_1,\cdots,f_{2n+3} \rangle.
        \end{align*}

        Then we have
        \begin{align*}
            Z^1(\mathfrak{D}_4,\Pic(\bar{X}))/\bZ\langle f_1,\cdots,f_{2n+3}\rangle=\bZ/2,
        \end{align*}
        where the nonzero element is induced by index set $I_1$.
        
        The element $f_{-1} \notin \bZ\langle f_1,\cdots,f_{2n+3}\rangle$,
        which implies
        \begin{align*}
            \bZ\langle f_{-1},\cdots,f_{2n+3}\rangle/\bZ\langle f_1,\cdots,f_{2n+3}\rangle=\bZ/2,
        \end{align*}
        and we have
        \begin{align*}
            \mathrm{H}^1(\mathfrak{D}_4,\Pic(\bar{X}))=0.
        \end{align*}
            
    \item [(18)]: The Sylow 2-subgroup $\mathfrak{D}_4$ is generated by $g_1,g_2,g_3$ where 
    \begin{align*}
        &g_1=c_1\dots c_{2n+1}c_{4n+3}(2,3)\dots(2n,2n+1),\\
	   &g_2=c_{2n+2}\dots c_{4n+2}c_{4n+3}(2n+3,2n+4)\dots(4n+1,4n+2),\\
	   &g_3=(1,2n+2)\dots (2n+1,4n+2).
    \end{align*}

    Under the action of $pr(\mathfrak{D}_4)$ on set $\{1,\cdots,4n+3\}$, there are $n+2$ orbits:
    \begin{align*}
        &O_1=\{1,2n+2\},\ O_{n+2}=\{4n+3\},\\
        &O_i=\{2i-2,2i-1,2n+2i-1,2n+2i\}, \quad i=2,\cdots,n+1,
    \end{align*}

    Over generator $g_1$, orbits $O_1$ and $O_{n+2}$ have signed permutation cycle with one minus sign, thus we only consider index sets
    \begin{align*}
        &I_1=O_1 \cup O_{n+2}=\{1,2n+2,4n+3\},\\
        &I_i=O_i=\{2i-2,2i-1,2n+2i-1,2n+2i\},\quad i=2,\cdots,n+1.
    \end{align*}

    Over generator $g_1$, index set $I_1$ has two signed permutation cycles with one minus sign, $c_1, c_{4n+3}$, then
    \begin{align*}
        \frac{1}{2}\sum_{i \in I_1}f_i \notin \bZ \langle f_1,\cdots,f_{4n+3}\rangle.
    \end{align*}
    For index sets $I_j$ with $j=2,\cdots,n+1$, we have
    \begin{align*}
        \frac{1}{2}\sum_{i \in I_j}f_i=f_{2j-2}+f_{2n+2j-1} \in \bZ \langle f_1,\cdots,f_{4n+3}\rangle.
    \end{align*}

    Then we have
        \begin{align*}
            Z^1(\mathfrak{D}_4,\Pic(\bar{X}))/\bZ\langle f_1,\cdots,f_{4n+3}\rangle=\bZ/2,
        \end{align*}
        where the nonzero element is induced by index set $I_1$.
        
        The element $f_{-1} \notin \bZ\langle f_1,\cdots,f_{4n+3}\rangle$, which implies
        \begin{align*}
            \bZ\langle f_{-1},\cdots,f_{4n+3}\rangle/\bZ\langle f_1,\cdots,f_{4n+3}\rangle=\bZ/2,
        \end{align*}
        and
        \begin{align*}
            \mathrm{H}^1(\mathfrak{D}_4,\Pic(\bar{X}))=0.
        \end{align*}

    \item [(22)]: The Sylow 2-subgroup $\mathfrak{D}_4$ is generated by $g_1,g_2,g_3$ where
    \begin{align*}
    	&g_1=c_1\dots c_{4n_1+2}(2,3)\dots(2n_1,2n_1+1)(2n_1+3,2n_1+4)\dots(4n_1+1,4n_1+2),\\
    	&g_2=c_{2n_1+2}\dots c_{4n_1+2n_2+3}(2n_1+3,2n_1+4)\dots(4n_1+1,4n_1+2)\\
    	&\times (4n_1+4,4n_1+5)\dots(4n_1+2n_2+2,4n_1+2n_2+3),\\
    	&g_3=(1,2n_1+2)\dots(2n_1+1,4n_1+2).
    \end{align*}

    Under the action of $pr(\mathfrak{D}_4)$ on set $\{1,\cdots,4n_1+2n_2+3\}$, there are $n_1+n_2+2$ orbits:
    \begin{align*}
        &O_1=\{1,2n_1+2\},\ O_{i}=\{2i-2,2i-1,2n_1+2i-1,2n_1+2i\},\quad i=2,\cdots,n_1+1,\\
        &O_{n_1+n_2+2}=\{4n_1+3\},\ O_{n_1+i+1}=\{4n_1+2i+2,4n_2+2i+3\},\quad i=1,\cdots,n_2.
    \end{align*}

    Over generator $g_2$, orbits $O_{1}, O_{n_1+n_2+2}$ have signed permutation cycles with one minus sign, thus we only consider index sets,
    \begin{align*}
        I_1=O_{1} \cup O_{n_1+n_2+2},\ I_i=O_{i}\quad i=2,\cdots,n_1+n_2+1.
    \end{align*}
    
    Over generator $g_1$, index set $I_1$ has two signed permutation cycles with one minus sign $c_1, c_{2n_1+2}$, then
    \begin{align*}
        \frac{1}{2}\sum_{i \in I_1}f_i \notin \bZ \langle f_1,\cdots,f_{4n_1+2n_2+3}\rangle
    \end{align*}

    For index sets $I_j$ with $j=2,\cdots,n_1+1$, we have
    \begin{align*}
        \frac{1}{2}\sum_{i \in I_j}f_i=f_{2j-2}+f_{2n_1+2j-1} \in \bZ \langle f_1,\cdots,f_{4n_1+2n_2+3}\rangle,
    \end{align*}
    and for index sets $I_{n_1+j+1}$ with $j=1,\cdots,n_2$, we have
    \begin{align*}
        \frac{1}{2}\sum_{i \in I_{n_1+j+1}}f_i=f_{4n_1+2j+2} \in \bZ \langle f_1,\cdots,f_{4n_1+2n_2+3}\rangle.
    \end{align*}

    Then we have
        \begin{align*}
            Z^1(\mathfrak{D}_4,\Pic(\bar{X}))/\bZ\langle f_1,\cdots,f_{4n_1+2n_2+3}\rangle=\bZ/2,
        \end{align*}
        where the nonzero element is induced by index set $I_1$.
        
        The element $f_{-1} \notin \bZ\langle f_1,\cdots,f_{4n_1+2n_2+3}\rangle$, which implies
        \begin{align*}
            \bZ\langle f_{-1},\cdots,f_{4n_1+2n_2+3}\rangle/\bZ\langle f_1,\cdots,f_{4n_1+2n_2+3}\rangle=\bZ/2,
        \end{align*}
        and
        \begin{align*}
            \mathrm{H}^1(\mathfrak{D}_4,\Pic(\bar{X}))=0.
        \end{align*}
    \item [(24)]: The Sylow 2-subgroup $\mathfrak{D}_4$ is generated by $g_1,g_2,g_3$ which is a special case of the previous class, class (22), with $n_1=n_2=n$.
    \end{itemize}
\end{example}

\begin{example}
\label{exam:case19}
    This example verifies the group $G=\mathfrak{C}_2^2:\mathfrak{F}_{p^r}$ described in case (19), satisfies the (H1) condition.

    The Sylow 2-subgroup is contained in subgroup $\mathfrak{C}_2^2:\mathfrak{C}_{p^r-1}$ generated by $g_1,g_3,g_4$ where
    \begin{align*}
        &g_1=c_1\dots c_{p^r+1}(1,lab_p(\Gamma),\dots,lab_p(\Gamma^{p^r-2}))(p^r+1,p^r+2),\\
    	&g_3=(p^r+1,p^r+2),\\
    	&g_4=c_{p^r+1}c_{p^r+2}.
    \end{align*}

    and we have
    \begin{align*}
        &g_1^{p^r-1}=g_2^2=g_3^2=id,\\
       &g_3g_4=g_4g_3,\\
       &g_1g_3g_1^{-1}=g_4g_3,\\
       &g_1g_4g_1^{-1}=g_4.
    \end{align*}
    
    Every group element $g\in \mathfrak{C}_2^2:\mathfrak{C}_{p^r-1}$ can be written as
    \begin{align*}
        g=g_1^ig_3^jg_4^k,\ 0\leq i \leq p^r-2,\ j,k \in \{0,1\}.
    \end{align*}
    
    Write $g_1=c_1c_{p^r+1}(p^r+1,p^r+2)\cdot\beta$ where $\beta$ is a product of signed permutation cycles with an even number of minus signs.
    
    We have
    \begin{align*}
        g_1^i=
        \begin{cases}
            \beta^i,\ &i \equiv 0 \mod 4,\\
            c_1c_{p^r+1}(p^r+1,p^r+2)\beta^i,\ &i \equiv 1 \mod 4,\\
            c_{p^r+1}c_{p^r+2}\beta^i,\ &i \equiv 2 \mod 4,\\
            c_1c_{p^r+2}(p^r+1,p^r+2)\beta^i,\ &i \equiv 3 \mod 4.
        \end{cases}
    \end{align*}
    
    Analysis for the other 3 cases $g_1^ig_3,g_1^ig_4,g_1^ig_3g_4$ are similar.

    A subgroup $H \subseteq \mathfrak{C}_2^2:\mathfrak{C}_{p^r-1}$ can be presented as 
    \begin{align*}
        H =\langle g_1^{i_1}g_3^{j_1}g_4^{k_1},\cdots,g_1^{i_T}g_3^{j_T}g_4^{k_T}\rangle.
    \end{align*}
    
    From analysis above, depending on those $T$ generators of $H$, there are 2 possibilities:
    \begin{enumerate}
        \item [(1)] If all generators $g_1^{i_t}g_3^{j_t}g_4^{k_t}=\beta_{t}^{i_t}$ for $1\leq t\leq T$, then
        \begin{align*}
            Z^1(H,\Pic(\bar{X}))/\bZ \langle f_1,\cdots,f_{p^r+2} \rangle=0
        \end{align*}
        and
        \begin{align*}
            \mathrm{H}^1(H,\Pic(\bar{X}))=0
        \end{align*}
        \item [(2)] Otherwise, there is only one index set $I$ such that
        \begin{align*}
            \frac{1}{2}\sum_{i \in I}f_i \in Z^1(H,\Pic(\bar{X}))\setminus \bZ \langle f_1,\cdots,f_{p^r+2}),
        \end{align*}
        where the index set $I$ is one of the sets below,
        \begin{align*}
             \{1,p^r+1\},\{1,p^r+2\},\{p^r+1,p^r+2\},\{1,p^r+1,p^r+2\}.
        \end{align*}
        This element is the nonzero element of
        \begin{align*}
            Z^1(H,\Pic(\bar{X}))/ \bZ \langle f_1,\cdots,f_{p^r+2}) = \bZ/2
        \end{align*}
        We also have
        \begin{align*}
            f_{-1} \notin \bZ \langle f_1,\cdots,f_{p^r+2}\rangle,
        \end{align*}
        and thus
        \begin{align*}
            \bZ \langle f_{-1},f_1,\cdots,f_{p^r+2}\rangle/\bZ \langle f_1,\cdots,f_{p^r+2}\rangle =\bZ/2,
        \end{align*}
        which implies
        \begin{align*}
            \mathrm{H}^1(H,\Pic(\bar{X}))=0.
        \end{align*}
    \end{enumerate}
    The group $\mathfrak{C}_2^2:\mathfrak{C}_{p^r-1}$ satisfies $\mathrm{(H1)}$ condition, so does the Sylow 2-subgroup of $G$ and so does the group $G$ described in class (19) below.
\end{example}

Up to conjugation, there are two classes of groups that form the building blocks of other classes:

\begin{itemize}
	\item [(1):] $\mathfrak D_{2n+1}$: the dihedral group of order $2(2n+1)$, with generators
	\begin{align*}
	g_1&=c_1c_2c_3\ldots c_{2n+2}(2,3)(4,5)\dots(2n,2n+1),\\
	g_2&=(1,2,4,\dots,2n,2n+1,2n-1,\dots,3).
	\end{align*}
	It has 2 orbits of length: $(2n+1)+1$.
        \newline

	\item[(2):] $\mathfrak F_{p^r}$: The Frobenius group of order $p^r(p^r-1)$, where $p$ is a prime integer greater than $2$.
	$$
	 \mathfrak F_{p^r}\cong AGL(1,p^r) \cong \mathbb{F}_{p^r} \rtimes \mathbb{F}^*_{p^r}.
	$$
	We may assume that $\mathbb{F}_{p^r}=(\mathbb{Z}/p)[x]/(h(x))$, 
	where $h$ is a degree $r$ irreducible monic polynomial, pick a representative of every class as $a_0+a_1x+\ldots+a_{r-1}x^{r-1}, a_i \in \{0,\ldots,p-1\}$ and label every class by 
	\[
	lab_p([a_0+a_1x+\ldots+a_{r-1}x^{r-1}])=
	\begin{cases}
	p^r, \quad \text{if }a_0=a_1=\ldots=a_{r-1}=0\\
	a_0+a_1p+\ldots+a_{r-1}p^{r-1}, \quad \text{otherwise}
	\end{cases}
	\]
	 
	Choosing a generator $\Gamma$ of the cyclic group $\mathbb{F}_{p^r}^*$. The order $p^r-1$ element of $\mathfrak{F}_{p^r}$ can by written as
	\[
	g_1=
	c_1\dots c_{p^r-1}c_{p^r}c_{p^r+1}
	\begin{pmatrix}
	1 & lab_p(\Gamma) & \dots & lab_p(\Gamma^{p^r-3}) & lab_p(\Gamma^{p^r-2})\\
	lab_p(\Gamma) & lab_p(\Gamma^2) & \dots & lab_p(\Gamma^{p^r-2}) & 1
	\end{pmatrix}.
	\]
	The order $p$ elements in $\mathfrak{F}_{p^r}$ has the form
	\begin{align*}
	g_2=
	\begin{pmatrix}
	p^r & 1 & \dots & p-2 & p-1\\
	1 & 2 & \dots & p-1 & p^r
	\end{pmatrix}
	\begin{pmatrix}
	p & p+1 & \dots & 2p-2 & 2p-1\\
	p+1 & p+2 & \dots & 2p-1 & p
	\end{pmatrix}\\
	\cdots
	\begin{pmatrix}
	p^r-p & p^r-p+1 & \dots & p^r-2 & p^r-1\\
	p^r-p+1 & p^r-p+2 & \dots & p^r-1 & p^r-p
	\end{pmatrix}.
	\end{align*}
 
	These two elements generate $\mathfrak{F}_{p^r}$.
	
	It has 2 orbits of length: $p^r+1$.
\end{itemize}

\begin{example}
     For dihedral group $\mathfrak{D}_{2n+1}$ of order $2(2n+1)$, the orbits are as follows:
     \begin{figure}[H]
    		\centering
    		\includegraphics[width=80mm]{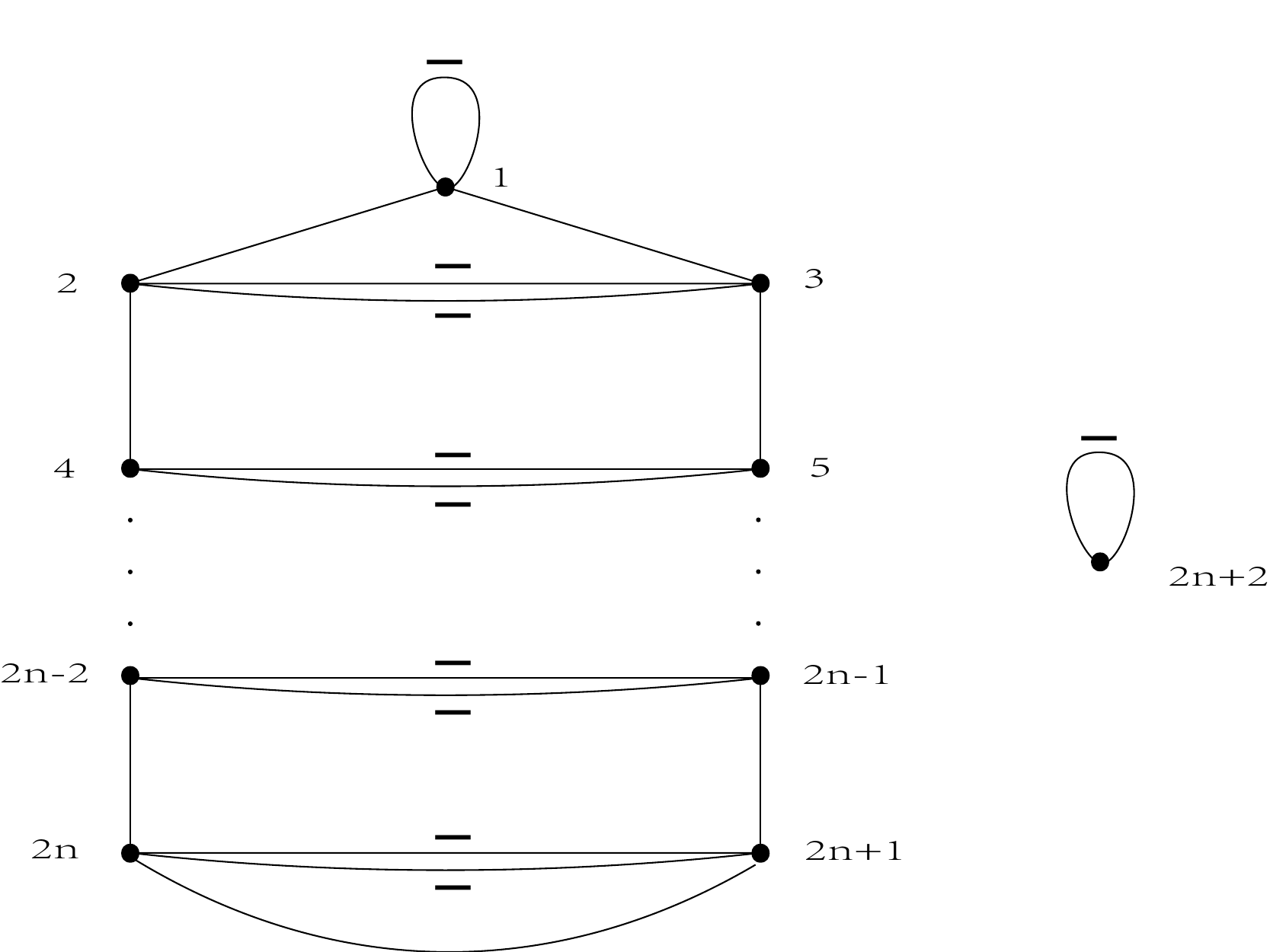}
    		\caption{Dihedral group $\mathfrak D_{2n+1}$ with orbits $(2n+1)+1$}
    	\end{figure}
    When $p=7, r=1$, and $\Gamma=3$, the group $\mathfrak F_7$ can be generated by
	\begin{align*}
	&g_1=c_1\ldots c_8(1,3,2,6,4,5),\\
	&g_2=(1,2,3,4,5,6,7).
	\end{align*}
    The orbits are as follows:
	\begin{figure}[H]
		\centering
		\includegraphics[width=80mm]{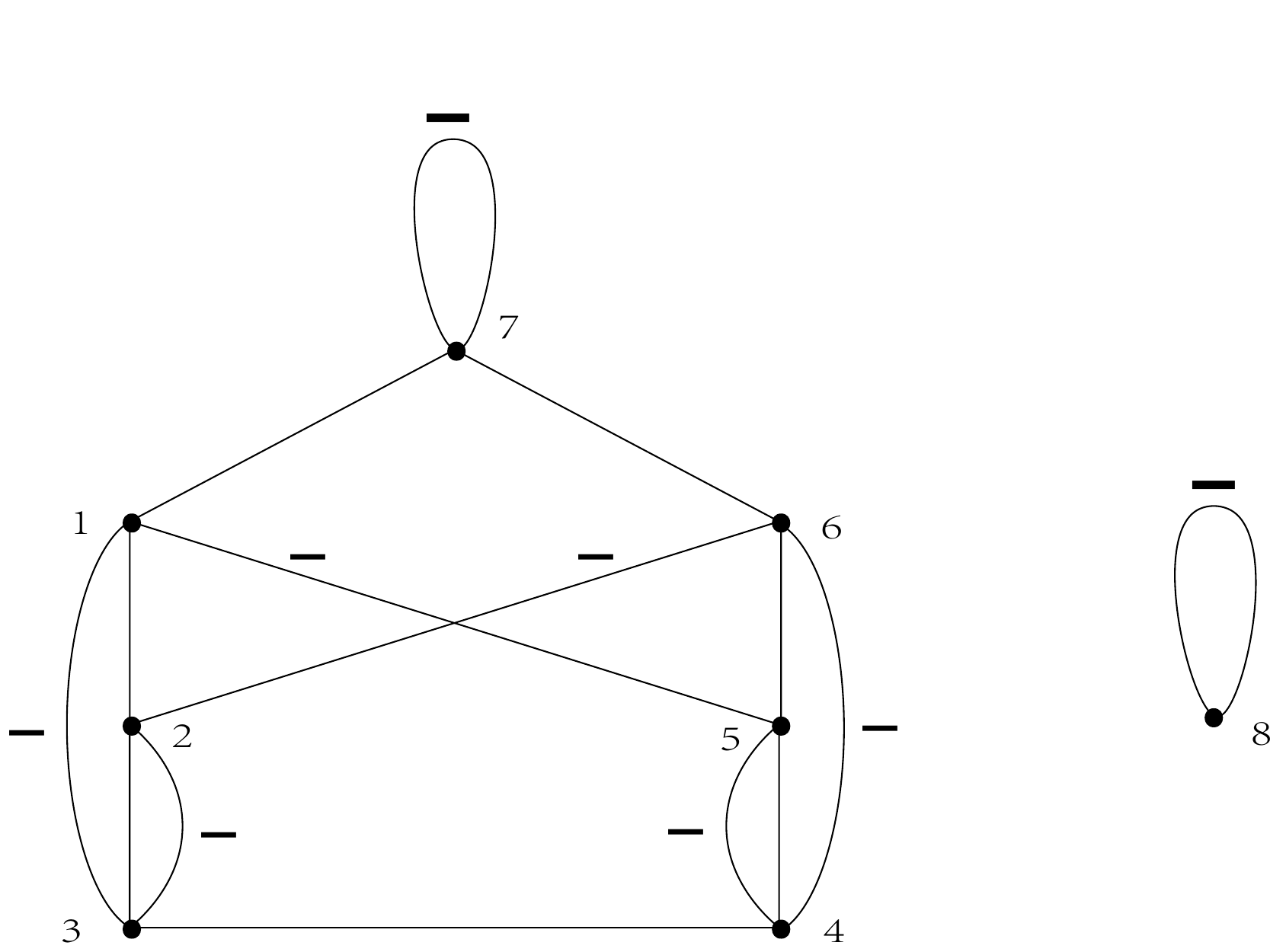}
		\caption{Frobenius group $\mathfrak F_7$ with orbits $7+1$}
	\end{figure}
    For these two classes of groups, their Sylow 2-subgroups are contained in the cyclic group $\langle g_1\rangle$, and they satisfy the (H1) condition and the relative minimality condition.  
\end{example}

\begin{theorem}
Let $G, G'$ be $\mathfrak{D}_{2n+1}$ or $\mathfrak{F}_{p^r}$, for positive integer $n,r$ and odd prime $p$, with generators $g_1,g_2$ for $G$ and $g_1',g_2'$ for $G'$, as above. If we identify the length one orbit of these groups, then $\mathfrak{D}_{2n_1+1}\times\mathfrak{D}_{2n_2+1}$ and $\mathfrak{D}_{2n+1}\times \mathfrak{F}_{p^r}$ satisfy $\mathrm{(H1)}$ 
and the relative minimality condition. 
\end{theorem}
\begin{proof}
The Sylow 2-subgroup of $G \times G'$ is contained in bicyclic group $\langle g_1 \rangle \times \langle g_1'\rangle$ and for any group element in $\langle g_1 \rangle \times \langle g_1'\rangle$, it has trivial group cohomology. By Proposition \ref{prop:h1bicyc}, the Sylow 2-subgroup of $G \times G'$ satisfies (H1) condition.

The relative minimality condition is easy to check.
\end{proof} 

\begin{remark}
\label{rem:14}
Let groups $G,G'$ be defined as above, with generators $g_1,g_2$ and $g_1',g_2'$, acting on $n_1+1$, respectively $n_2+1$, pairs of degenerate fibers. We have $G \times G' \subset W(\mathsf{D}_{n_1+n_2+1})$. 
Consider the subgroup $H \subset G \times G'$ with 3 generators $g_1g_1',g_2,g_2'$.
Since $G \times G'$ satisfies the $\mathrm{(H1)}$ condition, so does $H$.
All three generators of $H$ have trivial action on length one orbit, so 
$H$ leaves the length one orbit invariant, thus it does not satisfy the relative minimality condition.

This means the group representation contains a trivial representation. After eliminating the trivial representation, the group $H$, considered as a subgroup of $W(\mathsf{D}_{n_1+n_2})$,
satisfies both $\mathrm{(H1)}$ and relative minimality conditions.
\end{remark}

From the theorem and remark above, there are 8 classes of groups satisfying the (H1) and relative minimality conditions.

Assume that $g_1$ is the generator of $\mathfrak{D}_{2n_1+1}$ having order $2$(or the generator of $\mathfrak{F}_{p_1^{r_1}}$ having order $p_1^{r_1}-1$) and $g_2$ is the generator of $\mathfrak{D}_{2n_1+1}$ having order $2n_1+1$(or of $\mathfrak{F}_{p_1^{r_1}}$ having order $p_1^{r_1}$). $g_1'$ and $g_2'$ are generators of group $\mathfrak{D}_{2n_2+1}$( or $\mathfrak{F}_{p_2^{r_2}}$).
\begin{itemize}
	\item [(3):]$\mathfrak{D}_{2n_1+1}\times \mathfrak{D}_{2n_2+1}$
	
	The group is generated by $g_1,g_2,g_1',g_2'$.
	
	It has 3 orbits of length:$(2n_1+1)+(2n_2+1)+1$.
        \newline
        
	\item [(4):]$\mathfrak{D}_{2n+1}\times \mathfrak{F}_{p^r}$
	
	The group is generated by $g_1,g_2,g_1',g_2'$.
	
	It has 3 orbits of length: $(2n+1)+p^r+1$.
        \newline
        \item[(5):]$\mathfrak{C}_{2n_1+1}\rtimes\mathfrak{D}_{2n_2+1}$
	
	The group is generated by $g_1g_1',g_2,g_2'$.
	
	It has 2 orbits of length: $(2n_1+1)+(2n_2+1)$.
	\newline
        \item[(6):]$\mathfrak{C}_{2n+1}:\mathfrak{F}_{p^r}$
	
	The group is generated by $g_1g_1',g_2,g_2'$.
	
	It has 2 orbits of length: $(2n+1)+p^r$.
	\newline
        \item[(7):]$\mathfrak{C}_{2n_1+1}\rtimes\mathfrak{D}_{2n_2+1}$ or $\mathfrak{D}_{2n_1+1}$
 
	The group is generated by $g_1g_1',g_2g_2'$.

	It has 2 orbits of length: $(2n_1+1)+(2n_2+1)$.

        If $2n_1+1$ and $2n_2+1$ are coprime, this class is the same as class $(5)$, if $2n_2+1$ divides $2n_1+1$, then the group is $\mathfrak{D}_{2n_1+1}$
        \newline
        
	\item[(8):]$*$
	
	The group is generated by $g_1g_1',g_2g_2'$.
	
	It has 2 orbits of length: $(2n+1)+p^r$.

        If $2n+1$ and $p$ are coprime, this class is the same as class $(6)$.
\end{itemize}

In the following, we give more classes of groups that satisfy both (H1) and relative minimality conditions.
\begin{itemize}
	\item [(9):]$\mathfrak{D}_{4n+2}$
	
	The group is generated by
	\begin{align*}
	&g_1=c_1\dots c_{2n+1}c_{2n+2}(2,3)\dots(2n,2n+1),\\
	&g_2=(1,2,4,\dots,2n,2n+1,2n-1,\dots,3),\\
	&g_3=c_{2n+2}c_{2n+3}.	
	\end{align*}
 
	The Sylow 2-subgroup is $\mathfrak{C}_2\times \mathfrak{C}_2$, generated by $g_1,g_3$.
	
	It has 3 orbits of length: $(2n+1)+1+1$.
	\newline
        
        \item[(10):]$\mathfrak{C}_{2n+1}:\mathfrak{C}_4$
	
	The group is generated by
	\begin{align*}
	&g_1=c_1\dots c_{2n+1}c_{2n+2}(2,3)\dots (2n,2n+1)(2n+2,2n+3),\\
	&g_2=(1,2,4,\dots,2n,2n+1,2n-1,\dots,3).
	\end{align*}
 
	The Sylow 2-subgroup is $\mathfrak{C}_4$, generated by $g_1$.
	
	It has 2 orbits of length: $(2n+1)+2$.
	\newline
        \item[(11):]$\mathfrak{C}_{2n+1}:\mathfrak{D}_4$
	
	The group is generated by
	\begin{align*}
	&g_1=c_1\dots c_{2n+1}c_{2n+2}(2,3)\dots(2n,2n+1)(2n+2,2n+3),\\
	&g_2=(1,2,4,\dots,2n,2n+1,2n-1,\dots,3),\\
	&g_3=(2n+2,2n+3).
	\end{align*}
 
	The Sylow 2-subgroup is $\mathfrak{D}_4$, generated by $g_1,g_3$.
	
	It has 2 orbits of length: $(2n+1)+2$.
        \newline
        
	\item [(12):]$\mathfrak{D}_{4n+2}$
	
	The group is generated by
	\begin{align*}
	&g_1=c_1\dots c_{4n+2}(2,3)\dots(2n,2n+1)(2n+3,2n+4)\dots(4n+1,4n+2),\\
	&g_2=(1,2,4,\dots,2n,2n+1,2n-1,\dots,3)\\
 &\quad \times(2n+2,2n+3,2n+5,\dots,4n+1,4n+2,4n,\dots,2n+4),\\
	&g_3=(1,2n+2)\dots(2n+1,4n+2).
	\end{align*}
 
	The Sylow 2-subgroup is $\mathfrak{C}_2\times\mathfrak{C}_2$, generated by $g_1,g_3$.
	
	It has 1 orbit of length: $(4n+2)$.
	\newline
 
	\item [(13):]$(\mathfrak{D}_{2n+1})^2$
	
	The group is generated by 
	\begin{align*}
	&g_1=c_1\dots c_{4n+2}(2,3)\dots(2n,2n+1)(2n+3,2n+4)\dots(4n+1,4n+2),\\
	&g_2=(1,2,4,\dots,2n,2n+1,2n-1,\dots,3),\\
	&g_3=(2n+2,2n+3,\dots,4n+1,4n+2,\dots,2n+4),\\
	&g_4=(1,2n+2)\dots(2n+1,4n+2).
	\end{align*}
 
	The Sylow 2-subgroup is $\mathfrak{C}_2\times\mathfrak{C}_2$, generated by $g_1,g_4$.
	
	It has 1 orbit of length: $(4n+2)$.
        \newline
        
	\item [(14):]$\mathfrak{F}_{p^r}$ or $\mathfrak{C}_2 \times \mathfrak{F}_{p^r}$
	
	The group is generated by
	\begin{align*}
	&g_1=c_1\dots c_{p^r-1}c_{p^r}c_{p^r+1}(1,lab_p(\Gamma),\dots,lab_p(\Gamma^{p^r-2}))(p^r+1,p^r+2),\\
	&g_2=(p^r,1,\dots,p-1)(p,\dots,2p-1)\dots(p^r-p,\dots,p^r-1).
	\end{align*}
 
	The Sylow 2-subgroup is contained in the cyclic group generated by $g_1$.
	
	It has 2 orbits of length: $p^r+2$.
    
        If $4$ divides $p^r$, then the group is $\mathfrak{F}_{p^r}$, otherwise, the group is $\mathfrak{C}_2 \times \mathfrak{F}_{p^r}$.
        \newline
        
	\item[(15):]$(\mathfrak{C}_{2n+1} \rtimes \mathfrak{D}_{2n+1}) \rtimes \mathfrak{C}_2$
	
	The group is generated by
	\begin{align*}
	&g_1=c_1\dots c_{2n+1}c_{4n+3}(1,2n+2)(2,2n+3,3,2n+4)\dots(2n,4n+1,2n+1,4n+2),\\
	&g_2=(1,2,4,\dots,2n,2n+1,2n-1,\dots,3),\\
	&g_3=(2n+2,2n+3,2n+5,\dots,4n+1,4n+2,4n\dots,2n+4).
	\end{align*}
 
	The Sylow 2-subgroup is $\mathfrak{C}_4$, generated by $g_1$.
	
	It has 2 orbits of length: $(4n+2)+1$.
	\newline
 
	\item[(16):]$\mathfrak{C}_2 \times \mathfrak{F}_{p^r}$ or $\mathfrak{C}_{p^r} \rtimes \mathfrak{C}_{2(p^r-1)}$
	
	The group is generated by
	\begin{align*}
	&g_1=c_1\dots c_{p^r-1}c_{p^r}c_{p^r+1}(1,lab_p(\Gamma),\dots,lab_p(\Gamma^{p^r-2}))(p^r+1,p^r+2),\\
	&g_2=(p^r,1,\dots,p-1)(p,\dots,2p-1)\dots(p^r-p,\dots,p^r-1),\\
	&g_3=c_{p^r+1}c_{p^r+2}.
	\end{align*}
 
	The Sylow 2-subgroup is contained in $\langle g_1\rangle \times \langle g_3 \rangle$.
	
	It has 2 orbits of length: $p^r+2$.

        If $4$ divides $p^r-1$, then the group is $\mathfrak{C}_2 \times \mathfrak{F}_{p^r}$, otherwise, the group is $\mathfrak{C}_{p^r} \rtimes \mathfrak{C}_{2(p^r-1)}$.
        \newline
        
	\item[(17):]$\mathfrak{C}_2\times \mathfrak{F}_{p^r}$
	
	The group is generated by
	\begin{align*}
	&g_1=c_1\dots c_{p^r+1}(1,lab_p(\Gamma),\dots,lab_p(\Gamma^{p^r-2})),\\
	&g_2=(p^r,1,\dots,p-1)(p,\dots,2p-1)\dots(p^r-p,\dots,p^r-1),\\
	&g_3=c_{p^r+1}c_{p^r+2}.
	\end{align*}
 
	The Sylow 2-subgroup is contained in $\langle g_1\rangle \times \langle g_3\rangle$.
	
	It has 3 orbits of length: $p^r+1+1$.
        \newline
        
	\item[(18):]$\mathfrak{D}_{2n+1}\wr \mathfrak{C}_2$
	
	The group is generated by
	\begin{align*}
	&g_1=c_1\dots c_{2n+1}c_{4n+3}(2,3)\dots(2n,2n+1),\\
	&g_2=c_{2n+2}\dots c_{4n+2}c_{4n+3}(2n+3,2n+4)\dots(4n+1,4n+2),\\
	&g_3=(1,2n+2)\dots (2n+1,4n+2),\\
	&g_4=(1,2,4,\dots,2n,2n+1,2n-1,\dots,3),\\
	&g_5=(2n+2,2n+3,2n+5,\dots,4n+1,4n+2,4n,\dots,2n+4).
	\end{align*}
 
	The Sylow 2-subgroup is $\mathfrak{D}_4$, generated by $g_1,g_2,g_3$.
	
	It has 2 orbits of length: $(4n+2)+1$.
        \newline
	\item[(19):]$\mathfrak{C}_2^2:\mathfrak{F}_{p^r}$
	
	The group is generated by
	\begin{align*}
	&g_1=c_1\dots c_{p^r+1}(1,lab_p(\Gamma),\dots,lab_p(\Gamma^{p^r-2}))(p^r+1,p^r+2),\\
	&g_2=(p^r,1,\dots,p-1)(p,\dots,2p-1)\dots(p^r-p,\dots,p^r-1),\\
	&g_3=(p^r+1,p^r+2),\\
	&g_4=c_{p^r+1}c_{p^r+2}.
	\end{align*}
 
	The Sylow 2-subgroup is contained in $\mathfrak{C}_2^2:\mathfrak{C}_{p^r-1}$, generated by $g_1,g_3,g_4$.
	
	It has 2 orbits of length: $p^r+2$.
        \newline
 
	\item[(20):]$(\mathfrak{C}_{2n_1+1}\times \mathfrak{C}_{2n_2+1}\times \mathfrak{C}_{2n_3+1}):\mathfrak{C}_2^2$
	
	The group is generated by
	\begin{align*}
	&g_1=c_1\dots c_{2n_1+2n_2+2}(2,3)\dots(2n_1,2n_1+1)(2n_1+3,2n_1+4)\dots(2n_1+2n_2+1,2n_1+2n_2+2),\\
	&g_2=c_{2n_1+2}\dots c_{2n_1+2n_2+2n_3+3}(2n_1+3,2n_1+4)\dots(2n_1+2n_2+1,2n_1+2n_2+2)\\
	&\times (2n_1+2n_2+4,2n_1+2n_2+5)\dots(2n_1+2n_2+2n_3+2,2n_1+2n_2+2n_3+3),\\
	&g_3=(1,2,4,\dots,2n_1,2n_1+1,2n_1-1,\dots ,3),\\
	&g_4=(2n_1+2,2n_1+3,2n_1+5,\dots,2n_1+2n_2+1,2n_1+2n_2+2,2n_1+2n_2,\dots,2n_1+4),\\
	&g_5=(2n_1+2n_2+3,2n_1+2n_2+4,\dots,2n_1+2n_2+2n_3+2,2n_1+2n_2+2n_3+3,\dots,2n_1+2n_2+5).
	\end{align*}
 
	The Sylow 2-subgroup is $\mathfrak{C}_2\times \mathfrak{C}_2$, generated by $g_1,g_2$.
	
	It has 3 orbits of length: $(2n_1+1)+(2n_2+1)+(2n_3+1)$.
        \newline
 
	\item[(21):]$\mathfrak{C}_{2n_1+1}^2:(\mathfrak{C}_{2n_2+1}:\mathfrak{C}_4)$
	
	The group is generated by
	\begin{align*}
	&g_1=c_{2n_1+2}\dots c_{4n_1+2n_2+3}(1,2n_1+2)(2,2n_1+3,3,2n_1+4)\dots(2n_1,4n_1+2,2n_1+1,4n_1+2)\\
	&\times(4n_1+4,4n_1+5)\dots(4n_1+2n_3+2,4n_1+2n_3+3),\\
	&g_2=(1,2,4\dots,2n_1,2n_1+1,2n_1-1,\dots,3),\\
	&g_3=(2n_1+2,2n_1+3,\dots,4n_1+1,4n_1+2,\dots, 2n_1+4),\\
	&g_4=(4n_1+3,4n_1+4,\dots,4n_1+2n_2+2,4n_1+2n_2+3,\dots,4n_1+5).
	\end{align*}
 
	The Sylow 2-subgroup is $\mathfrak{C}_4$, generated by $g_1$.
	
	It has 2 orbits of length: $(4n_1+2)+(2n_2+1)$.
        \newline
	\item[(22):]$(\mathfrak{C}_{2n_1+1}^2\times\mathfrak{C}_{2n_2+1}):\mathfrak{D}_4$
	
	The group is generated by
	\begin{align*}
	&g_1=c_1\dots c_{4n_1+2}(2,3)\dots(2n_1,2n_1+1)(2n_1+3,2n_1+4)\dots(4n_1+1,4n_1+2),\\
	&g_2=c_{2n_1+2}\dots c_{4n_1+2n_2+3}(2n_1+3,2n_1+4)\dots(4n_1+1,4n_1+2)\\
	&\times (4n_1+4,4n_1+5)\dots(4n_1+2n_2+2,4n_1+2n_2+3),\\
	&g_3=(1,2n_1+2)\dots(2n_1+1,4n_1+2),\\
	&g_4=(1,2,4,\dots,2n_1,2n_1+1,2n_1-2,\dots,3),\\
	&g_5=(2n_1+2,2n_1+3,\dots,4n_1+1,4n_1+2,\dots,2n_1+4),\\
	&g_6=(4n_1+3,4n_1+4,\dots,4n_1+2n_2+2,4n_1+2n_2+3,\dots,4n_1+5).
	\end{align*}
 
	The Sylow 2-subgroup is $\mathfrak{D}_4$, generated by $g_1,g_2,g_3$.
	
	It has 2 orbits of length: $(4n_1+2)+(2n_2+1)$.
	\newline
        \item[(23):]$\mathfrak{C}_{2n+1}^3:\mathfrak{C}_2^2:\mathfrak{C}_3$
	
	The group is generated by
	\begin{align*}
	&g_1=c_1\dots c_{4n+2}(2,3)\dots(2n,2n+1)(2n+3,2n+4)\dots(4n+1,4n+2),\\
	&g_2=c_{2n+2}\dots c_{6n+3}(2n+3,2n+4)\dots(4n+1,4n+2)(4n+4,4n+5)\dots(6n+2,6n+3),\\
	&g_3=(1,2,4,\dots,2n,2n+1,2n-1,\dots,3),\\
	&g_4=(2n+2,2n+3,2n+5,\dots,4n+1,4n+2,4n,\dots, 2n+4),\\
	&g_5=(4n+3,4n+4,4n+6,\dots,6n+2,6n+3,6n+1,\dots,4n+5),\\
	&g_6=(1,2n+2,4n+3)\dots(2n+1,4n+2,6n+3).
	\end{align*}
 
	The Sylow 2-subgroup is $\mathfrak{C}_2\times \mathfrak{C}_2$, generated by $g_1,g_2$.
	
	It has 1 orbit of length: $(6n+3)$.
	\newline
        \item[(24):]$\mathfrak{C}_{2n+1}^3:\mathfrak{D}_4:\mathfrak{C}_3$
	
	The group is generated by
	\begin{align*}
	&g_1=c_1\dots c_{4n+2}(2,3)\dots(2n,2n+1)(2n+3,2n+4)\dots(4n+1,4n+2),\\
	&g_2=c_{2n+2}\dots c_{6n+3}(2n+3,2n+4)\dots(4n+1,4n+2)(4n+4,4n+5)\dots(6n+2,6n+3),\\
	&g_3=(1,2n+2)\dots(2n+1,4n+2),\\
	&g_4=(1,2,4,\dots,2n,2n+1,2n-1,\dots,3),\\
	&g_5=(2n+2,2n+3,2n+5,\dots,4n+1,4n+2,4n,\dots, 2n+4),\\
	&g_6=(4n+3,4n+4,4n+6,\dots,6n+2,6n+3,6n+1,\dots,4n+5),\\
	&g_7=(1,2n+2,4n+3)\dots(2n+1,4n+2,6n+3).
	\end{align*}
 
	The Sylow 2-subgroup is $\mathfrak{D}_4$, generated by $g_1,g_2,g_3$.
	
	It has 1 orbit of length: $(6n+3)$.
\end{itemize}

\section{Numerical results}
\label{sect:num}
In this section, we report on {\tt Magma} computations to identify relevant subgroups of $W(\mathsf{D}_n)$ satisfying $\rm (H1)$ condition, for $n=4,\ldots, 9$. For more details on orbit decomposition and orbit stabilizer group, please refer to \cite{YangWeb-conic}. 

\begin{table}[H]
	\begin{tabular}{|c|c|c|c|c|c|}
		\hline
		\multicolumn{3}{|c|}{$W(\mathsf{D}_4)$} & \multicolumn{3}{c|}{$W(\mathsf{D}_5)$} \\ \hline
	Group Name & Class Number & Condition & Group Name & Class Number & Condition\\ \hline
	D4(1):$\mathfrak{S}_3$& 1 & $n=1$ & D5(1):$\mathfrak{D}_6$ & 9 & $n=1$\\ \hline
		&       &       & D5(2):$\mathfrak{C}_3:\mathfrak{C}_4$ & 10 & $n=1$\\ \hline
		&       &       & D5(3):$\mathfrak{C}_3:\mathfrak{D}_4$ & 11 & $n=1$\\ \hline
	\end{tabular}
\end{table}
\begin{table}[H]
    \small
	\begin{tabular}{|c|c|c|c|c|c|}
		\hline
		\multicolumn{3}{|c|}{$W(\mathsf{D}_6)$} & \multicolumn{3}{c|}{$W(\mathsf{D}_7)$} \\ \hline
	Group Name & Class Number & Condition & Group Name & Class Number & Condition\\ \hline
	D6(1):$\mathfrak{S}_3$ & 7 & $n_1=n_2=1$ & 
    D7(1):$\mathfrak{C}_5\rtimes \mathfrak{C}_4$ & 10 & $n=2$\\ \hline
    D6(2):$\mathfrak{D}_5$ & 1 & $n=2$ &
    D7(2):$\mathfrak{D}_{10}$ & 9 & $n=2$\\ \hline
	D6(3):$\mathfrak{D}_6$ & 12 & $n=2$ &
    D7(3):$\mathfrak{F}_5$ & 14 & $p=5,r=1$\\
  \hline
    D6(4):$\mathfrak{D}_6$ &       &      &
	D7(4):$\mathfrak{S}_3^2$ & 3 & $n_1=n_2=1$\\ \hline
    D6(5):$\mathfrak{C}_3\rtimes\mathfrak{S}_3$	& 5 & $n_1=n_2=1$ & 
	D7(5):$(\mathfrak{C}_3\rtimes \mathfrak{S}_3)\rtimes \mathfrak{C}_2$ & 15 & $n=1$\\ \hline
    D6(6):$\mathfrak{F}_5$ & 2 & $p=5,r=1$ & 
	D7(6):$\mathfrak{C}_2\times \mathfrak{F}_5$ & 17 & $p=5,r=1$\\ \hline
	D6(7):$\mathfrak{S}_4$ &       &      & D7(7):$\mathfrak{C}_5\rtimes \mathfrak{D}_4$ & 11 & $n=2$\\ \hline
	D6(8):$\mathfrak{S}_4$ &       &      & D7(8):$\mathfrak{C}_2\times \mathfrak{F}_5$ & 14 or 16 & $p=5,r=1$\\ \hline
	D6(9):$\mathfrak{S}_4$ &       &      & D7(9):$\mathfrak{S}_3\wr \mathfrak{C}_2$ & 18 & $n=1$\\ \hline
    D6(10): $\mathfrak{S}_4$ &     &     & D7(10):$\mathfrak{C}_2^2\rtimes \mathfrak{F}_5$ & 19 & $p=5,r=1$\\ \hline
	D6(11):$\mathfrak{S}_3^2$ & 13 & $n=1$ & & & \\ \hline
	D6(12):$\mathfrak{C}_2\times \mathfrak{S}_4$ &       &      &       &       &      \\ \hline
	D6(13):$\mathfrak{C}_2\times \mathfrak{S}_4$	&       &      &       &       &      \\ \hline
	D6(14):$\mathfrak{S}_5$	&       &      &       &       &      \\ \hline
	D6(15):$\mathfrak{S}_5$	&       &      &       &       &      \\ \hline
	\end{tabular}
\end{table}
\begin{table}[H]
    \footnotesize
	\begin{tabular}{|c|c|c|c|c|c|}
		\hline
		\multicolumn{3}{|c|}{$W(\mathsf{D}_8)$} & \multicolumn{3}{c|}{$W(\mathsf{D}_9)$} \\ \hline
	Group Name & Class Number & Condition & Group Name & Class Number & Condition\\ \hline
	D8(1):$\mathfrak{F}_7$ & 2 & $p=7,r=1$ & D9(1):$\mathfrak{C}_7\rtimes \mathfrak{C}_4$ & 10 & $n=3$\\ \hline
	D8(2):$\mathfrak{D}_{15}$	& 5 or 7 & $n_1=1,n_2=2$ & D9(2):$\mathfrak{D}_{14}$ & 9 & $n=3$\\ \hline
	D8(3):$\mathfrak{D}_7$ & 1 & $n=3$ & D9(3):$\mathfrak{D}_{14}\rtimes \mathfrak{C}_2$ & 11 & $n=3$\\ \hline
	D8(4):$\mathfrak{C}_3\rtimes \mathfrak{F}_5$ & 6 or 8 & $n=1,p=5,r=1$ & D9(4):$\mathfrak{S}_3\times \mathfrak{D}_5$ & 3 & $n_1=1,n_2=2$\\ \hline
		&       &       & D9(5):$\mathfrak{C}_7 \rtimes \mathfrak{C}_{12}$ & 16 & $p=7,r=1$\\ \hline
		&       &       & D9(6):$\mathfrak{C}_2\times \mathfrak{F}_7$ & 17 & $p=7,r=1$\\ \hline
		&       &       & D9(7):$\mathfrak{C}_3^3 \rtimes \mathfrak{C}_2^2$ & 20 & $n_1=n_2=n_3=1$\\ \hline
		&       &       & D9(8):$\mathfrak{C}_3^2\rtimes (\mathfrak{C}_3\rtimes \mathfrak{C}_4)$ & 21 & $n_1=n_2=1$\\ \hline
		&       &       & D9(9):$\mathfrak{S}_3\times \mathfrak{F}_5$ & 4 & $n=1,p=5,r=1$\\ \hline
		&       &       & D9(10):$\mathfrak{C}_2^2\rtimes \mathfrak{F}_7$ & 19 & $p=7,r=1$\\ \hline
		&       &       & D9(11):$\mathfrak{S}_3^2\rtimes \mathfrak{S}_3$ & 22 & $n_1=n_2=1$\\ \hline
		&       &       & D9(12):$\mathfrak{C}_3^3\rtimes \mathfrak{C}_2^2\rtimes \mathfrak{C}_3$ & 23 & $n=1$\\ \hline
		&       &       & D9(13):$\mathfrak{C}_3^3.\mathfrak{S}_4$ & 24 & $n=1$\\ \hline
	\end{tabular}
\end{table}

\begin{remark}
    The case D4(1): $\fS_3$ corresponds to the family of stably rational, nonrational conic bundles over $\bQ$ described in \cite{StabRatNonRat}.
\end{remark}

\section{Further Questions}
\begin{enumerate}
	\item We have D6(6):$\mathfrak{S}_4 \subset W(\mathsf{D}_6)$ and we guess that case D6(6) lies in the class generated by
	\begin{align*}
	&g_1=c_1\cdots c_{4n+2}(2,3)\cdots(2n,2n+1)(2n+3,2n+4)\cdots(4n+1,4n+2),\\
	&g_2=(1,2,\cdots,2n,2n+1,\cdots,3)(2n+2,2n+3,\cdots,4n+1,4n+2,\cdots,2n+4),\\
	&g_3=(1,2n+2)(2,2n+3)(4,2n+5)\cdots(2n,4n+1),\\
	&g_4=(1,2n+2)(3,2n+4)(5,2n+6)\cdots(2n+1,4n+2).
	\end{align*}
	
	When $n=1$, the group is $\mathfrak{S}_4$, the Sylow 2-subgroup is $\mathfrak{D}_4$. When $n=2$, the group is $\mathfrak{C}_2\times \mathfrak{C}_2^4\rtimes \mathfrak{D}_5$, the Sylow 2-subgroup is $\mathfrak{C}_2\times \mathfrak{C}_2^2\wr\mathfrak{C}_2$. When $n=3$, the group is $\mathfrak{C}_2^6.\mathfrak{D}_7$, the Sylow 2-subgroup is $\mathfrak{C}_2^6.\mathfrak{C}_2$.
	
	As we vary parameter $n$, the order of the Sylow 2-subgroup gets larger and the Sylow 2-subgroup has more complicated structure.
	
	We need more tools and theorems to show it satisfies the (H1) condition.
	
    \item Find more classes to contain the rest cases of $W(\sf D_6)$.
    \item Find stably $k$-rational but not $k$-rational conic bundle $X$ of degree $d=8-n \leq 3$ such that Galois group of the splitting field of $X$ is described in section \ref{sect:num}.
\end{enumerate}

\bibliographystyle{plain}
\bibliography{conic}
\end{document}